\documentclass[a4paper,reqno]{amsart}
\usepackage[utf8]{inputenc}
\usepackage[T1]{fontenc}
\usepackage[british]{babel}
\usepackage{amsmath,amssymb,amsthm}
\allowdisplaybreaks[3]
\usepackage{xcolor}
\usepackage[all]{xy}
\usepackage{mathtools}
\usepackage{hyperref}
\usepackage{stmaryrd}
\usepackage{tikz-cd}
\usepackage{mathbbol}
\usepackage{mathabx}

\renewcommand{\binom}{\genfrac{\lgroup}{\rgroup}{0pt}{1}}
\newcommand{\verticalbinomial}{\genfrac{|}{|}{0pt}{1}}
\newcommand{\LieAlg}{\mathbf{Lie}_R}
\newcommand{\Grp}{\mathbf{Grp}}

\newcommand{\XMod}{\mathbf{XMod}}
\newcommand{\Pt}{\mathbf{Pt}}
\newcommand{\Act}{\mathbf{Act}}
\newcommand{\A}{\mathbb{A}}

\newcommand{\SH}{{\rm (SH)}}

\newcommand{\noproof}{\hfill\qed}

\renewcommand{\MR}[1]{}

\theoremstyle{plain} 
\newtheorem{thm}[subsection]{Theorem}
\newtheorem{cor}[subsection]{Corollary} 
\newtheorem{lemma}[subsection]{Lemma} 
\newtheorem{prop}[subsection]{Proposition}

\theoremstyle{definition} 
\newtheorem{defi}[subsection]{Definition}

\theoremstyle{remark}
\newtheorem{rmk}[subsection]{Remark}
\newtheorem{ex}[subsection]{Example}

\newdir{>>}{{}*!/3.5pt/:(1,-.2)@^{>}*!/3.5pt/:(1,+.2)@_{>}*!/7pt/:(1,-.2)@^{>}*!/7pt/:(1,+.2)@_{>}}
\newdir{ >>}{{}*!/8pt/@{|}*!/3.5pt/:(1,-.2)@^{>}*!/3.5pt/:(1,+.2)@_{>}}
\newdir{ |>}{{}*!/-3.5pt/@{|}*!/-8pt/:(1,-.2)@^{>}*!/-8pt/:(1,+.2)@_{>}}
\newdir{ >}{{}*!/-8pt/@{>}}
\newdir{>}{{}*:(1,-.2)@^{>}*:(1,+.2)@_{>}}
\newdir{<}{{}*:(1,+.2)@^{<}*:(1,-.2)@_{<}}

\title{Compatible actions in semi-abelian categories}

\author{Davide di Micco}
\address[Davide di Micco]{Università degli Studi di Milano, Via Saldini 50, 20133 Milano, Italy}
\email[Davide di Micco]{davide.dimicco@unimi.it}

\author{Tim Van~der Linden}
\address[Tim Van~der Linden]{Institut de
Recherche en Math\'ematique et Physique, Universit\'e catholique
de Louvain, che\-min du cyclotron~2 bte~L7.01.02, B--1348
Louvain-la-Neuve, Belgium}
\email[Tim Van~der Linden]{tim.vanderlinden@uclouvain.be}

\thanks{The second author is a Research
Associate of the Fonds de la Recherche Scientifique--FNRS}

\date{\today}

\begin{document}
\begin{abstract}
The concept of a pair of compatible actions was introduced in the case of groups by Brown and Loday~\cite{BL84} and in the case of Lie algebras by Ellis~\cite{Ell91}. In this article we extend it to the context of semi-abelian categories (that satisfy the \emph{Smith is Huq} condition). We give a new construction of the Peiffer product, which specialises to the definitions known for groups and Lie algebras. We use it to prove our main result, on the connection between pairs of compatible actions and pairs of crossed modules over a common base object. We also study the Peiffer product in its own right, in terms of its universal properties, and prove its equivalence with existing definitions in specific cases.
\end{abstract}



\maketitle

\section*{Introduction}

The concept of a pair of compatible actions was first introduced in the category of groups by Brown and Loday, in relation to their work on the non-abelian tensor product of groups~\cite{BL84}. Later, in~\cite{Ell91}, the definition was adapted to the context of Lie algebras, where it was further studied in~\cite{Khm99,DiM19}. Since then, several other particular instances of compatible actions have been defined, in various settings: see for example~\cite{Gne99,CaKhPR16,CaKhPR17}. The aim of this article is to provide a general definition in semi-abelian categories (in the sense of~\cite{JMT02}), in a way that extends these as special cases. In particular this will give us the tools to develop a unified theory, in such a way that computing the non-abelian tensor product of compatible actions is the same as computing the non-abelian tensor product of internal crossed modules. This process generalises the diverse particular notions of non-abelian tensor product that appear in the literature so far.

With this idea in mind, we first examine the case of groups from a new perspective, aiming to use a diagrammatic and internal approach whenever this is possible. To do so, we take advantage of the equivalence between group actions in the usual sense and internal actions (introduced in~\cite{Bourn-Janelidze:Semidirect, BJK05}) in the category $\Grp$, as well as the equivalence between crossed modules of groups and internal crossed modules in $\Grp$ (see~\cite{Jan03}). Thus we may separate properties which are specific for groups from those that are purely categorical.

The conditions which we single out in the category $\Grp$ in terms of the internal action formalism become our general definition of \lq\lq a pair of compatible actions\rq\rq. This definition makes sense as soon as the surrounding category is semi-abelian. However, we shall always assume that the so-called \emph{Smith is Huq condition} \SH\ holds as well. This is a relatively mild condition which excludes loops, for instance, but includes all categories of groups with operations; see~\cite{MFVdL12, CGVdL15b}. This simplifies our work, since under \SH\ internal crossed modules allow an easier description~\cite{Jan03, MFVdL12}.

Our main tool is an extension, to the semi-abelian context, of the \emph{Peiffer product} $M\bowtie N$ of two objects $M$ and $N$ acting on each other (via an action $\xi^N_M$ of $N$ on $M$ and an action $\xi^M_N$ of $M$ on $N$). A notion of Peiffer product has already been introduced in~\cite{CMM17}, in the special case of a pair of internal precrossed modules over a common base object. Ours, however, is a different approach, and a priori the two notions do not coincide. Our definition is a direct generalisation of the group and Lie algebra versions of the Peiffer product, which were originally introduced respectively in~\cite{Whi41} and in~\cite{Ell91}. It is well defined as soon as the two objects $M$ and $N$ act on each other, whereas for the definition in~\cite{CMM17} they also need to satisfy some compatibility conditions. Moreover, when the actions $\xi^N_M$ and $\xi^M_N$ are compatible, the Peiffer product $M\bowtie N$ is endowed with internal crossed module structures $(M\xrightarrow{l_M}M\bowtie N,\xi_M^{M\bowtie N})$ and $(N\xrightarrow{l_N}M\bowtie N,\xi_N^{M\bowtie N})$.

We use this as an ingredient in the generalisation of a result, stated in~\cite{BL84} for groups and in~\cite{Khm99} for Lie algebras, to any semi-abelian category that satisfies the condition \SH. We show namely that two objects $M$ and $N$ act on each other compatibly if and only if there exists a third object $L$ endowed with two internal crossed module structures $(M\xrightarrow{\mu}L,\xi^L_M)$ and $(N\xrightarrow{\nu}L,\xi^L_N)$. Amongst other things, this allows us to deduce that our definition of compatibility for pairs of internal actions restricts to the classical definitions for groups and Lie algebras. Another consequence of this result is that the non-abelian tensor product introduced in the forthcoming article~\cite{dMVdL19.3} has two natural interpretations: either as a tensor product of compatible internal actions, or equivalently as a tensor product of crossed modules over a common base object.

Finally, we study the Peiffer product via its universal properties. We also prove that, under the additional hypothesis of algebraic coherence~\cite{CGVdL15b}, our definition of the Peiffer product coincides with the one given in~\cite{CMM17}.

\subsection*{Structure of the text}
The paper is organised as follows. 
In the first section we collect preliminary definitions and results on internal actions in semi-abelian categories. We recall the definitions of the bifunctors $\flat$ and $\diamond$ as well as some related constructions. For a given semi-abelian category $\A$, we describe the category of points in $\A$ and the category of internal actions in $\A$, together with the equivalence between the two. 

In Section~\ref{Section groups} we examine the concept of a pair of compatible actions in the category of groups. First we consider the definition of compatibility given in~\cite{BL84} and we translate it into its diagrammatic form. Then we construct the Peiffer product as a coequaliser and we prove that it coincides with the definition already known for the case of groups. In Proposition~\ref{prop:remark 2.16 in BL} we prove a result stated in~\cite{BL84}, namely that two groups $M$ and $N$ act on each other compatibly if and only if there exists a third group $L$ endowed with two crossed module structures $(M\xrightarrow{\mu}L,\xi^L_M)$ and $(N\xrightarrow{\nu}L,\xi^L_N)$.

Section~\ref{Section Semiabelian} contains our main results. We work in the context of a semi\-abelian category $\A$ that satisfies the \emph{Smith is Huq condition} \SH. We express the definition of compatibility in this general context and show in Proposition~\ref{prop:coterminal crossed module induce compatible actions} that whenever we have a pair of internal crossed modules over a common base object, they induce a pair of compatible internal actions. 

Then we construct the Peiffer product of two internal actions in three distinct ways: first we imitate what happens in the case of groups, constructing the Peiffer product for each pair of objects acting on each other. In Proposition~\ref{prop:Peiffer is a pushout of semidirect products} we prove that this is the same as taking the pushout of the two semi-direct products. Then we give a more specific definition that requires the actions to be compatible. Finally, we show in Proposition~\ref{prop:q=q_S} that, if the compatibility conditions are satisfied, then the two definitions coincide. 

We prove in Proposition~\ref{prop:crossed module structures of the Peiffer product} that whenever the actions are compatible, their Peiffer product is automatically endowed with internal crossed module structures $(M\xrightarrow{l_M}M\bowtie N,\xi_M^{M\bowtie N})$ and $(N\xrightarrow{l_N}M\bowtie N,\xi_N^{M\bowtie N})$. This leads to Theorem~\ref{thm:compatible actions iff pair of coterminal crossed module}, which is a generalisation to semi-abelian categories of Proposition~\ref{prop:remark 2.16 in BL}, proven for groups in the previous section: two objects $M$ and $N$ act on each other compatibly if and only if there exists a third object $L$ endowed with two internal crossed module structures $(M\xrightarrow{\mu}L,\xi^L_M)$ and $(N\xrightarrow{\nu}L,\xi^L_N)$. Via this result we obtain Corollary~\ref{cor:in grp the internal defi of compatible action is the same as the particular one} and Corollary~\ref{cor:in LieAlg the internal defi of compatible action is the same as the particular one} as confirmations of the equivalence between our general definition of compatibility and the specific ones in the cases of groups and Lie algebras.

We conclude the paper with a study of the Peiffer product via its universal properties (Section~\ref{Section Peiffer}, in particular Proposition~\ref{Universal property} and Proposition~\ref{prop:Peiffer is a pushout of semidirect products}). Here we also prove that, under the additional hypothesis of algebraic coherence~\cite{CGVdL15b}, our definition of the Peiffer product coincides with the one given in~\cite{CMM17}. Via results in~\cite{CMM17}, this further implies that under an additional condition called (UA), the actions induced by two $L$-crossed module structures have a Peiffer product which is again an $L$-crossed module; furthermore, it is the coproduct in $\XMod_L(\A)$ of the given $L$-crossed modules. This generalises Proposition~3.4 in~\cite{DiM19}.

\section{Preliminaries on internal actions}
In what follows, we let $\A$ be a semi-abelian category~\cite{JMT02}: pointed, Barr exact, Bourn protomodular with binary coproducts. This concept unifies earlier attempts (including, for instance,~\cite{Huq, Gerstenhaber, Orzech}) at providing a categorical framework that extends the context of abelian categories to include non-additive categories of algebraic structures such as groups, Lie algebras, loops, rings, etc. In this setting, the basic lemmas of homological algebra---the \emph{$3\times 3$ Lemma}, the \emph{Short Five Lemma}, the \emph{Snake Lemma}---hold~\cite{Bou01,BB04}. 

The category of internal actions $\Act(\A)$ and the category of internal crossed modules $\XMod(\A)$ in any semi-abelian category $\A$ are again semi-abelian. We recall their definitions.

\begin{defi}\cite{Bou03}
A \emph{regular pushout} is a commutative square of regular epimorphisms as on the left
\[
\xymatrix@R=5ex@C=3em{
A \ar[r]^-f \ar[d]_-\alpha &
 B \ar[d]^-\beta \\
A' \ar[r]_-{f'} &
 B'
}
\qquad\qquad
\xymatrix@R=5ex@C=3em{
A'\times_{B'}B \ar[r]^-{\pi_{B}} \ar[d]_-{\pi_{A'}} &
 B \ar[d]^-\beta \\
A' \ar[r]_-{f'} &
 B'
}
\]
such that the comparison map $\langle\alpha,f\rangle\colon A\to A'\times_{B'}B$ to the induced pullback square on the right is a regular epimorphism as well. 
\end{defi}

It is well known that in a semi-abelian category, a commutative square of regular epimorphisms is a regular pushout if and only if it is a pushout. In fact this characterises semi-abelian categories amongst finitely cocomplete homological categories (in the sense of~\cite{BB04}: pointed, regular, protomodular). Regular pushouts can be recognised as follows:

\begin{lemma}\cite{Bou03}
\label{lemma:ker of reg pushout is reg epi}
Consider a square of regular epimorphisms in a homological category and take kernels to the left as in the diagram
\[
\xymatrix{
K_f \ar@{.>}[d]_-{k} \ar@{{ |>}.>}[r]^-{k_f} & A \ar@{-{ >>}}[r]^-{f} \ar@{-{ >>}}[d]_-{\alpha} & B \ar@{-{ >>}}[d]^-{\beta}\\
K_{f'} \ar@{{ |>}.>}[r]_-{k_{f'}} & A' \ar@{-{ >>}}[r]_-{f'} & B'.
}
\]
The induced map $k$ is a regular epimorphism if and only if the given square is a regular pushout.\noproof
\end{lemma}

\subsection{The bifunctor \texorpdfstring{$\flat$}{-b-}}
For an object $A$ in a semi-abelian category $\A$, internal $A$-actions are defined as algebras over a certain monad $A\flat(-)$. 

\begin{defi}
The bifunctor $\flat\colon \A\times\A\to\A$ is defined on objects as the kernel
\[
\xymatrix{
A\flat B \ar@{{ |>}->}[r]^-{k_{A,B}} & A+B \ar[r]^-{\binom{1_A}{0}} & A.
}
\]
Using the universal property of kernels, its behaviour on arrows is determined by 
\[
\xymatrix{
A\flat B \ar@{.>}[d]_-{f\flat g} \ar@{{ |>}->}[r]^-{k_{A,B}} & A+B \ar[d]_-{f+g} \ar[r]^-{\binom{1_A}{0}} & A \ar[d]^-{f}\\
A'\flat B' \ar@{{ |>}->}[r]_-{k_{A',B'}} & A'+B' \ar[r]_-{\binom{1_{A'}}{0}} & A'.
}
\]
\end{defi}

\begin{ex}
\label{ex:bemolle in Grp}
In the category $\Grp$ the coproduct $A+B$ is the so-called \emph{free product} of $A$ and $B$, the group freely generated by the disjoint union of $A$ and $B$, modulo the relations that hold in $A$ or in $B$. This means that an element in $A+B$ can be represented as a word obtained by juxtaposition of elements in $A$ and in $B$. Then it is easy to deduce that $A\flat B$ is the subgroup of $A+B$ whose elements are represented by the words of the form $a_1b_1\cdots a_nb_n$ such that $a_1\cdots a_n=1\in A$. Furthermore, it can be shown that each word in $A\flat B$ can be written as a juxtaposition of formal conjugations, that is
\[
A\flat B=\langle aba^{-1}\mid a\in A, b\in B\rangle.
\]
The following example expresses the idea of the proof, which easy generalises to any word in $A\flat B$.
\begin{align*}
a_1b_1a_2b_2a_3b_3&=(a_1b_1a_1^{-1})(a_1a_2b_2a_2^{-1}a_1^{-1})(a_1a_2a_3)b_3\\
&=(a_1b_1a_1^{-1})(a_1a_2b_2(a_1a_2)^{-1})1(1b_31^{-1})
\end{align*}
\end{ex}

\begin{rmk}
For any fixed object $A\in\A$, the triple $(A\flat(-),\eta^A,\mu^A)$ is a monad, where for $A$, $B\in\A$ we define $\eta^A_B\colon B\to A\flat B$ as in 
\[
\xymatrix{
B \ar@{.>}[d]_-{\eta^A_B} \ar[rd]^-{i_B}\\
A\flat B \ar@{{ |>}->}[r]_-{k_{A,B}} & A+B \ar[r]^-{\binom{1_A}{0}} & A
}
\]
and $\mu^A_B\colon A\flat(A\flat B)\to A\flat B$ as in
\[
\xymatrix{
A\flat(A\flat B) \ar@{.>}[d]_-{\mu^A_B} \ar@{{ |>}->}[r]^-{k_{A,A\flat B}} & A+(A\flat B) \ar[d]^-{\binom{i_A}{k_{A,B}}} \ar[r]^-{\binom{1_A}{0}} & A \ar@{=}[d]\\
A\flat B \ar@{{ |>}->}[r]_-{k_{A,B}} & A+B \ar[r]_-{\binom{1_A}{0}} & A.
}
\]
\end{rmk}

\begin{lemma}
\label{lemma:regular epis are bemolle-stable}
In a semi-abelian category, consider regular epimorphisms $\alpha\colon A\to A'$ and $\beta\colon B\to B'$. Then both $\alpha+\beta$ and $\alpha\flat\beta$ are regular epimorphisms as well.
\end{lemma}
\begin{proof}
The first statement is easily shown checking that, if $\alpha=coeq(x_1,x_2)$ and $\beta=coeq(y_1,y_2)$, then $\alpha+\beta=coeq(x_1+y_1,x_2+y_2)$. For what regards the second statement we build the diagram
\[
\xymatrixcolsep{3pc}
\xymatrix{
A\flat B \ar@{{ |>}->}[r]^-{k_{A,B}} \ar[d]_-{\alpha\flat\beta} & A+B \ar@{-{ >>}}[d]^-{\alpha+\beta} \ar@{-{ >>}}[r]^-{\binom{1_A}{0}} & A \ar@{-{ >>}}[d]^-{\alpha}\\
A'\flat B' \ar@{{ |>}->}[r]_-{k_{A',B'}} & A'+B' \ar@{-{ >>}}[r]_-{\binom{1_{A'}}{0}} & A'.
}
\]
Thanks to Lemma~\ref{lemma:ker of reg pushout is reg epi} it suffices to show that the right-hand square is a pushout in order to obtain that $\alpha\flat\beta$ is a regular epimorphism as well. This is easy to do by direct verification of the universal property of pushouts.
\end{proof}

\subsection{The cosmash product \texorpdfstring{$\diamond$}{}}
Cosmash products~\cite{Smash} may be used to define commutators~\cite{MM10, HVdL11} and may help expressing properties of internal actions. We start by exploring the relationship with $\flat$.

\begin{defi}
Given two objects $A$ and $B$ in $\A$, consider the map 
\[
\Sigma_{A,B}=\binom{\langle 1_A,0\rangle}{\langle0,1_B\rangle}=\langle \binom{1_A}{0},\binom{0}{1_B}\rangle\colon A+B \longrightarrow A\times B.
\]
Since $\A$ is semi-abelian, the morphism $\Sigma_{A,B}$ is a regular epimorphism. By taking its kernel we find the short exact sequence
\[
\xymatrixcolsep{3pc}
\xymatrix{
0 \ar[r] & A\diamond B \ar@{{ |>}->}[r]^-{h_{A,B}} & A+B \ar[r]^-{\Sigma_{A,B}} & A\times B \ar[r] & 0
}
\]
where $A\diamond B$ is called the \emph{cosmash product} of $A$ and $B$.
\end{defi}

\begin{rmk}
\label{rmk:cosmash and bemolle}
Notice that the inclusion of $A\diamond B$ into $A+B$ factors through $A\flat B$, because the latter is the kernel of $\binom{1_A}{0}\colon A+B\to A$. Moreover we have another split short exact sequence involving the cosmash product, namely
\begin{equation}\label{Cosmash as kernel of bemol}
\xymatrix{
0 \ar[r] & A\diamond B \ar@{{ |>}->}[r]^-{i_{A,B}} & A\flat B \ar@<2pt>[r]^-{\tau^A_B} & B \ar[r] \ar@<2pt>[l]^-{\eta^A_B} & 0
}
\end{equation}
where $\tau^A_B\coloneq \binom{0}{1_B}\circ k_{A,B}$ is the trivial action of $A$ on $B$.
This can be seen by constructing the $3\times 3$ diagram
\[
\xymatrix{
& 0 \ar[d] & 0 \ar[d] & 0 \ar[d]\\
0 \ar[r] & A\diamond B \ar@{.>}[rd]|-{h_{A,B}} \ar@{{ |>}->}[r]^-{i_{A,B}} \ar@{{ |>}->}[d]_-{i_{B,A}} & A\flat B \ar@{-{ >>}}[r]^-{\tau^A_B} \ar@{{ |>}->}[d]^-{k_{A,B}} & B \ar[r] \ar@{=}[d] & 0\\
0 \ar[r] & B\flat A \ar@{{ |>}->}[r]_-{k_{B,A}} \ar@{-{ >>}}[d]_-{\tau^B_A} & A+B \ar@{-{ >>}}[r]^-{\binom{0}{1_B}} \ar@{-{ >>}}[d]_-{\binom{1_A}{0}} & B \ar[r] \ar[d] & 0\\
0 \ar[r] & A \ar@{=}[r] \ar[d] & A \ar[r] \ar[d] & 0 \ar[r] \ar[d] & 0\\
& 0 & 0 & 0
}
\]
from the bottom-right square by taking kernels, and then by noticing that the top-left object is the kernel of the comparison morphism from $A+B$ to the pullback induced by the lower-right square: since this morphism is precisely $\Sigma_{A,B}$, its kernel is $A\diamond B$.

Moreover the upper left square is a pullback and hence $A\diamond B$ can be seen as the intersection of the subobjects $A\flat B$ and $B\flat A$ of $A+B$. Furthermore, since $\A$ is semi-abelian, in the split short exact sequence~\eqref{Cosmash as kernel of bemol} the morphisms $i_{A,B}$ and $\eta^A_B$ are jointly extremally epimorphic. Thus we obtain the regular epimorphism
\[
\xymatrix{
(A\diamond B)+B \ar@{-{ >>}}[r]^-{\binom{i_{A,B}}{\eta^A_B}} & A\flat B.
}
\]
\end{rmk}

\begin{lemma}
\label{lemma:-bX preserves coequalisers of reflexive graphs}
Let $X$ be an object in a semi-abelian category $\A$. Then the functor $(-)\flat X\colon\A\to\A$ preserves coequalisers of reflexive graphs.
\end{lemma}
\begin{proof}
Consider a reflexive graph with its coequaliser
\[
\xymatrixcolsep{3pc}
\xymatrix{
A \ar@<4pt>[r]^-{d} \ar@<-4pt>[r]_-{c} & B \ar[l]|-{e} \ar@{-{ >>}}[r]^-{q} & Q
}
\]
and the induced diagram
\[
\xymatrix{
A\diamond X \ar@{{ |>}->}[d]_-{i_{A,X}} \ar@<2pt>[r]^-{d\diamond 1_X} \ar@<-2pt>[r]_-{c\diamond 1_X} & B\diamond X \ar@{{ |>}->}[d]_-{i_{B,X}} \ar@{-{ >>}}[r]^-{q\diamond 1_X} & Q\diamond X \ar@{{ |>}->}[d]^-{i_{Q,X}}\\
A\flat X \ar@{-{ >>}}[d]_-{\tau^A_X} \ar@<2pt>[r]^-{d\flat 1_X} \ar@<-2pt>[r]_-{c\flat 1_X} & B\flat X \ar@{-{ >>}}[d]_-{\tau^B_X} \ar@{-{ >>}}[r]^-{q\flat 1_X} & Q\flat X \ar@{-{ >>}}[d]^-{\tau^Q_X}\\
X \ar@<2pt>[r]^-{1_X} \ar@<-2pt>[r]_-{1_X} & X \ar@{-{ >>}}[r]_-{1_X} & X
}
\]
By using Corollary~2.27 in~\cite{HVdL11v1} we know that $q\diamond 1_X$ is again the coequaliser of $d\diamond 1_X$ and $c\diamond 1_X$. We already know that $q\flat 1_X$ is a regular epimorphism by Lemma~\ref{lemma:regular epis are bemolle-stable} and that $(q\flat 1_X)\circ (d\flat 1_X)=(q\flat 1_X)\circ (c\flat 1_X)$, so it remains to show the universal property. 

First of all, by examining the squares on the right, we can see that they form a horizontal morphism of vertical short exact sequences, and since $1_X$ is an isomorphism, we conclude that the top square is a pullback. This implies that it is also a pushout: indeed when we take kernels horizontally we obtain an induced isomorphism between them, which in turn implies that the given square is a pushout.

Now suppose that there exists a morphism $z\colon B\flat X\to Z$ such that $z\circ (d\flat 1_X)=z\circ (c\flat 1_X)$. Then $z\circ i^A_X\circ (d\diamond 1_X)=z\circ i^A_X\circ (c\diamond 1_X)$ and hence there is a unique morphism $\phi\colon Q\diamond X\to Z$ such that $\phi\circ (q\diamond 1_X)=z\circ i^A_X$. Our claim now follows from the universal property of the pushout.
\end{proof}

\subsection{The ternary cosmash product}
Following~\cite{Higgins}, in~\cite{HVdL11} Hartl and Van der Linden define the $n$-ary version of the cosmash product. We are interested in the ternary case, and in some relations between it and the binary case.

\begin{defi}
Given three objects $A$, $B$ and $C$ in $\A$, consider the map 
\[
\Sigma_{A,B,C}=
\begin{pmatrix}
i_A & i_A & 0 \\
i_B & 0 & i_B \\
0 & i_C & i_C
\end{pmatrix}
\colon A+B+C \longrightarrow (A+B)\times(A+C)\times(B+C).
\]
Its kernel is written
\[
\xymatrix{
A\diamond B\diamond C \ar@{{ |>}->}[rr]^-{h_{A,B,C}} && A+B+C
}
\]
and it is called the \emph{ternary cosmash product} of $A$, $B$ and $C$. Like in the binary case, it is obvious that, up to isomorphism, the ternary cosmash product does not depend on the order of the objects.
\end{defi}
In~\cite{HVdL11} the authors define folding operations linking cosmash products of different arities: for our purposes we only need to recall one of them.
\begin{defi}
Given two objects $A$ and $B$ we can construct a map 
\[
S^{A,B}_{2,1}\colon A\diamond A\diamond B\to A\diamond B
\]
through the diagram
\[
\xymatrixcolsep{3pc}
\xymatrix{
A\diamond A\diamond B \ar@{.>}[d]_-{S^{A,B}_{2,1}} \ar@{{ |>}->}[r]^-{h_{A,A,B}} & A+A+B \ar[d]^-{\binom{1_A}{1_A}+1_B} \ar[r]^-{\Sigma_{A,A,B}} & (A+A)\times(A+B)\times(A+B) \ar[d]^-{\binom{1_A}{1_A}\times \bigl(\binom{0}{1_B}\circ \pi_i\bigr)}\\
A\diamond B \ar@{{ |>}->}[r]_-{h_{A,B}} & A+B \ar[r]_-{\Sigma_{A,B}} & A\times B.
}
\]
\end{defi}

We need a map between $(A+B)\flat C$ and the ternary cosmash product $A\diamond B\diamond C $.
\begin{defi}
Consider the object $(A+B)\flat C$ and define the map $j_{A,B,C}$ as in the diagram
\[
\xymatrix@!0@R=3.5em@C=12em{
A\diamond B\diamond C \ar@{.>}[d]_-{j_{A,B,C}} \ar@{{ |>}->}[rd]^-{h_{A,B,C}}\\
(A+B)\flat C \ar@{{ |>}->}[r]_-{k_{(A+B),C}} & A+B+C \ar[r]^-{\binom{1_{A+B}}{0}} \ar[rd]_-{\Sigma_{A,B,C}} & A+B\\
&& (A+B)\times(A+C)\times(B+C). \ar[u]_-{\pi_1}
}
\]
In particular, if $A=B$, then we have the commutative diagram
\[
\xymatrixcolsep{3pc}
\xymatrix{
A\diamond A\diamond C \ar@{{ |>}->}@/^2pc/[rr]^-{h_{A,A,B}} \ar[d]_-{S^{A,C}_{2,1}} \ar[r]^-{j_{A,A,C}} & (A+A)\flat C \ar@{{ |>}->}[r]^-{k_{(A+A),C}} \ar[d]^-{\binom{1_A}{1_A}\flat 1_C} & A+A+C \ar[d]^-{\binom{1_A}{1_A}+1_C}\\
A\diamond C \ar@{{ |>}->}@/_2pc/[rr]_-{h_{A,B}} \ar@{{ |>}->}[r]_-{i_{A,C}} & A\flat C \ar@{{ |>}->}[r]_-{k_{A,C}} & A+C.
}
\]
\end{defi}

\begin{lemma}
\label{lemma:covering the bemolle of a coproduct with some cosmash products}
It is possible to cover the object $(A+B)\flat C$ with the three components $(A\diamond B\diamond C)$, $(A\flat C)$ and $(B\flat C)$.
\end{lemma}
\begin{proof}
By Lemma~2.12 in~\cite{HVdL11} we know that there is a regular epimorphism of the form
\[
\xymatrix{
(A\diamond B\diamond C)+(A\diamond C)+(B\diamond C) \ar@{-{ >>}}[rr]^-{e} && (A+B)\diamond C.
}
\]
Using Remark~\ref{rmk:cosmash and bemolle} we are able to construct the square
\[
\scalebox{0.9}{
\xymatrixrowsep{3pc}
\xymatrixcolsep{5pc}
\xymatrix{
(A\diamond B\diamond C)+(A\diamond C)+(B\diamond C)+C+C \ar@{-{ >>}}[r]^-{1+\binom{i_{A,C}}{\eta^A_C}+\binom{i_{B,C}}{\eta^B_C}} \ar@{-{ >>}}[d]_-{e+\binom{1_C}{1_C}} & (A\diamond B\diamond C)+(A\flat C)+(B\flat C) \ar@{-{>}}[d]^-{
 \left\lgroup
 \begin{smallmatrix}
 j_{A,B,C}\\
 i_A\flat 1_C\\
 i_B\flat 1_C	
 \end{smallmatrix}
 \right\rgroup
}\\
((A+B)\diamond C) + C \ar@{-{ >>}}[r]_-{\binom{i_{(A+B),C}}{\eta^{A+B}_C}} & (A+B)\flat C
}
}
\]
from which we see that the vertical map on the right is a regular epimorphism.
\end{proof}

\subsection{The categories \texorpdfstring{$\Pt(\A)$}{Pt(A)} and \texorpdfstring{$\Act(\A)$}{Act(A)}}

In semi-abelian categories there is a concept of internal action, which via a semi-direct product construction is equivalent to the concept of a point---a split epimorphism with a chosen splitting. 

\begin{defi}
A \emph{point} $(p,s)$ in $\A$ is a split epimorphism $p$ with a chosen splitting~$s$, that is $p\colon A\to B$ and $s\colon {B\to A}$ such that $p\circ s=1_B$. A morphism of points $(p,s)\to (p',s')$ is given by a pair of vertical maps $(f,g)$ such that the two squares formed by parallel morphisms
\[
\xymatrix{
A \ar[d]_-{f} \ar@<2pt>[r]^-{p} & B \ar@<2pt>[l]^-{s} \ar[d]^-{g}\\
A' \ar@<2pt>[r]^-{p'} & B' \ar@<2pt>[l]^-{s'}
}
\]
commute. $\Pt(\A)$ is the category of points in $\A$ and morphisms between them. Since the codomain of $p$ is $B$, the point $(p,s)$ is also called a \emph{point over $B$}.
\end{defi}

Having described the category of points, we now shift to internal actions, whose category is equivalent to the former whenever the base category $\A$ is semi-abelian.

\begin{defi}
An \emph{internal action of $A$ on $X$} (or simply \emph{$A$-action} or \emph{action}) in~$\A$ is a triple $(A,X,\xi)$ with $\xi\colon A\flat X\to X$ a map in $\A$ such that $(X,\xi)$ is an algebra for the monad $A\flat(-)\colon\A\to\A$. A morphism of actions from $(A,X,\xi)$ to $(A',X',\xi')$ is given by a pair $(f,g)$ of maps in $\A$, with $f\colon A\to A'$ and $g\colon X\to X'$, such that the following diagram commutes:
\[
\xymatrix{
A\flat X \ar[r]^{f\flat g} \ar[d]_-{\xi} & A'\flat X' \ar[d]^-{\xi'}\\
X \ar[r]_-{g} & X'
} 
\]
The category of actions and morphisms between them is denoted by $\Act(\A)$.
\end{defi}

\begin{ex}
If we fix $\A=\Grp$ we find that internal actions coincide with the usual group actions. Indeed due to Example~\ref{ex:bemolle in Grp}, in order to define such an internal action $\xi\colon A\flat X\to X$ it suffices to specify where the elements of the form $axa^{-1}$ are sent, since they generate the whole subgroup $A\flat X$. Now an internal action~$\xi$ corresponds to the group action $\psi\colon A\times X\to X$ given by $\psi(a,x)\coloneq \xi(axa^{-1})$. Conversely, starting from a group action $\psi$ we define $\xi\colon A\flat X\to X$ on the generators by $\xi(axa^{-1})\coloneq \psi(a,x)$. It is easy to show that these are actions in the appropriate sense. ($\xi$ being a morphism and the axioms for it to be an internal action amount to the group action axioms for the function $\psi$.) The correspondence just depicted determines an equivalence between internal actions in $\Grp$ and group actions.
\end{ex}

\begin{rmk}
Whenever the base category $\A$ is semi-abelian we have an equivalence of categories $\Pt(\A)\simeq\Act(\A)$. The functor $\Pt(\A)\to\Act(\A)$ sends a point $(p\colon {A\to B}, s\colon {B\to A})$ to the action $(B,K_p,\xi)$, where $\xi$ is the unique morphism making the diagram 
\begin{align*}
\xymatrix{
B\flat K_p \ar@{.>}[d]_-{\xi} \ar[r]^-{k_{B,K_p}} & B+K_p \ar[d]^-{\binom{s}{k_p}} \ar[r]^-{\binom{1_B}{0}} & B\ar@{=}[d] \\
K_p \ar@{{ |>}->}[r]_-{k_p} & A \ar[r]_-{p} & B
}
\end{align*}
commute. The functor $\Act(\A)\to \Pt(\A)$ sends an action $(A,X,\xi)$
 to the point
\[
\xymatrix{
X\rtimes_{\xi} A \ar@<2pt>[r]^-{\pi_\xi} & A \ar@<2pt>[l]^-{i_\xi}
}
\]
where the \emph{semi-direct product} $X\rtimes_{\xi} A$ is defined as the coequaliser
\[
\xymatrixcolsep{4pc}
\xymatrixrowsep{3pc}
\xymatrix{
A\flat X \ar@<4pt>[r]^-{i_X\circ\xi} \ar@<-1pt>[r]_-{k_{A,X}} & A+X \ar[r]^{\sigma_{\xi}} & X\rtimes_{\xi} A,
}
\]
the map $\pi_\xi$ is the unique map such that 
\[
\xymatrix{
A+X \ar[r]^-{\sigma_{\xi}} \ar[rd]_-{\pi_{A,X}} & X\rtimes_{\xi} A \ar@{.>}[d]^-{\pi_{\xi}}\\
& A
}
\]
commutes, and finally $i_\xi=\sigma_\xi\circ i_A$. We will sometimes write $X\rtimes_{\xi} A$ as $X\rtimes A$, when there is no risk of confusion regarding the action involved.
Notice that the map 
\[
k\coloneq \sigma_{\xi}\circ i_X\colon X\to X\rtimes_{\xi} A 
\]
is the kernel of $\pi_{\xi}$: it is easy to see that $\pi_{\xi}\circ k=0$, whereas for the universal property some work needs to be done---see, for instance, \cite{MFS12}. 
\end{rmk}

\begin{ex}
The \emph{trivial action} $(A,X,\tau)$ is given by 
\[
\tau=\binom{0}{1_X}\circ k_{A,X}\colon A\flat X\to X.
\]
Then we have that $(X\rtimes_{\tau} A,\sigma_{\tau}) \cong Coeq(i_X\circ(\binom{0}{1_X}\circ k_{A,X}),k_{A,X})$. Both $\binom{1_A}{0}$ and $\binom{0}{1_X}$ coequalise these two maps, so (following the example of the trivial action in~$\Grp$) a first guess would be that 
\[
Coeq(i_X\circ\binom{0}{1_X}\circ k_{A,X},k_{A,X})\cong (A\times X,\langle\binom{1_A}{0},\binom{0}{1_X}\rangle). 
\]
In order to prove this, we may use the equivalence $\Pt(\A)\simeq \Act(\A)$. In particular we claim that the desired point is given by $(\pi_A\colon A\times X\to A, \langle 1_A,0\rangle\colon A\to A\times X)$ and hence it suffices to show that $\tau=\binom{0}{1_X}\circ k_{A,X}$ makes the diagram
\begin{align*}
\xymatrix{
A\flat X \ar[d]_-{\tau} \ar[r]^-{k_{A,X}} & A+X \ar[d]^-{\binom{\langle 1_A,0\rangle}{\langle 0,1_X\rangle}} \\
X \ar[r]_-{\langle 0,1_X\rangle} & A\times X
}
\end{align*}
commute. This is done by direct and easy calculations.
\end{ex}

\begin{ex}
The \emph{conjugation action} $(A,A,\chi_A)$ is given by
\[
\chi_A=\binom{1_A}{1_A}\circ k_{A,A}\colon A\flat A\to A.
\]
Then we have that $(A\rtimes_{\chi_A} A,\sigma_{\chi_A}) \cong Coeq(i_2\circ(\binom{1_A}{1_A}\circ k_{A,A}),k_{A,A})$. Both $\binom{1_A}{0}$ and $\binom{1_A}{1_A}$ coequalise these two maps, so a first guess (again following the example of~$\Grp$) would be that 
\[
Coeq(i_2\circ\binom{1_A}{1_A}\circ k_{A,A},k_{A,A})\cong (A\times A,\langle\binom{1_A}{0},\binom{1_A}{1_A}\rangle). 
\]
In order to prove this, we use the same strategy as in the previous example.
\end{ex}

\begin{rmk}
Notice that, by the definition of the semi-direct product, it is easy to show that the diagram on the left
\[
\xymatrix{
A\flat X \ar@{{ |>}->}[r]^-{k_{A,X}} \ar[d]_-{\xi} & A+X \ar[d]^-{\sigma_{\xi}}\\
X \ar[r]_-{k_{\pi_{\xi}}} & X\rtimes_{\xi}A
}
\qquad\qquad
\xymatrixcolsep{4pc}
\xymatrix{
A\flat X \ar[d]_-{i_{\xi}\flat k_{\pi_{\xi}}} \ar[r]^-{\xi} & X \ar[d]^-{k_{\pi_{\xi}}}\\
(X\rtimes_{\xi}A)\flat(X\rtimes_{\xi}A) \ar[r]_-{\chi_{(X\rtimes_{\xi}A)}} & X\rtimes_{\xi}A
}
\]
is a pushout. Thanks to this commutativity we can show that also the square on the right commutes, which means that \lq\lq computing an action\rq\rq\ is the same as \lq\lq computing the conjugation in the induced semi-direct product\rq\rq.
\end{rmk}

\begin{rmk}
\label{rmk:xi(f flat 1) is an action}
Notice that, if $(B,X,\xi\colon B\flat X\to X)$ is an action and $f\colon A\to B$ is any map, then also $(A,X,\xi\circ(f\flat 1_X)\colon A\flat X\to X)$ is an action. Indeed the diagrams
\begin{align*}
\xymatrix{
X \ar@{=}[d] \ar[r]^-{\eta^A_X} & A\flat X \ar[d]^-{f\flat 1_X}\\
X \ar@{=}[rd] \ar[r]^-{\eta^A_X} & B\flat X \ar[d]^-{\xi}\\
& X
}
&&
\xymatrix{
A\flat(A\flat X) \ar[rr]^-{\mu^A_X} \ar[rd]|-{f\flat(f\flat 1_X)} \ar[d]_-{1_A\flat(f\flat 1_X)} && A\flat X \ar[d]^-{f\flat 1_X}\\
A\flat(B\flat X) \ar[r]^-{f\flat 1_{B\flat X}} \ar[d]_-{1_A\flat\xi} \ar[rd]|-{f\flat\xi} & B\flat(B\flat X) \ar[d]^-{1_B\flat\xi} \ar[r]^-{\mu^B_X} & B\flat X \ar[d]^-{\xi}\\
A\flat X \ar[r]_-{f\flat 1_X} & B\flat X \ar[r]_-{\xi} & X
}
\end{align*}
commute. The action $\xi\circ(f\flat 1_X)$ is often called \emph{pullback action} of $\xi$ along $f$ and the reason is the following. Consider the diagram
\[
\xymatrixcolsep{3pc}
\xymatrix{
X \ar@{=}[d] \ar@{{ |>}->}[r]^-{k_{\pi_{\xi'}}} & X\rtimes_{\xi'} A \ar[d]|-{1_X\rtimes f} \ar@<2pt>[r]^-{\pi_{\xi'}} & A \ar[d]^-{f} \ar@<2pt>[l]^-{\sigma_{\xi'}}\\
X \ar@{{ |>}->}[r]_-{k_{\pi_{\xi}}} & X\rtimes_{\xi} B \ar@<2pt>[r]^-{\pi_{\xi}} & B \ar@<2pt>[l]^-{\sigma_{\xi}}
}
\]
where the bottom row is the point associated to $\xi$, whereas the first row is obtained taking the pullback of $\pi_{\xi}$ along $f$. Then it is easy to see that the action $\xi'$ coincides with $\xi\circ(f\flat 1_X)$.
\end{rmk}

\begin{rmk}
In order to recover a point over $B$, in general slightly less is needed than a $B\flat(-)$-algebra structure. Every time we have an action $\xi\colon A\flat X\to X$ we can construct the corresponding \emph{action core} ${}^{\diamond}\xi\colon A\diamond X\to X$ as the composition of $\xi$ and $i_{A,X}\colon {A\diamond X\to A\flat X}$. Action cores (maps $A\diamond X\to X$ that satisfy suitable axioms) were defined and studied in~\cite{HVdL11, HL13}. The main point is that, in the semi-abelian context, from an action core ${}^{\diamond}\xi\colon A\diamond X\to X$ we can recover the action $\xi$. Furthermore, crossed module structures can be expressed in terms of action cores.
\end{rmk}

\begin{ex}
Consider an action $\xi\colon A\flat X\to X$ in $\Grp$, sending each generator $axa^{-1}$ of $A\flat X$ to ${}^{a}x\in X$. In order to understand how the action core ${}^{\diamond}\xi\colon A\diamond X\to X$ looks, we first need to make explicit what the inclusion $i_{A,X}\colon A\diamond X\to A\flat X$ does. It is easy to see that $A\diamond X$ is the subgroup of $A+X$ generated by the commutators, that is the words of the form $axa^{-1}x^{-1}$ with $a\in A$ and $x\in X$. The map $i_{A,X}$ sends a generator $axa^{-1}x^{-1}$ to $\left(axa^{-1}\right)\left(1x^{-1}1^{-1}\right)$. This means that the action core ${}^{\diamond}\xi$ sends an element of the form $axa^{-1}x^{-1}$ to $\xi\bigl(\bigl(axa^{-1}\bigr)\bigl(1x^{-1}1^{-1}\bigr)\bigr)={}^{a}xx^{-1}$.
\end{ex}

Our last ingredient is the definition of an internal crossed module in a semi-abelian category $\A$. Internal crossed modules are equivalent to internal categories; the conditions that make this happen were obtained in~\cite{Jan03}. In order to have a description which is as simple as possible, we require that $\A$ satisfies the condition \SH: further details on this definition (and on its general version which does not require \SH) can be found in~\cite{Jan03,HVdL11,MFVdL12}. Let us just add here that the crossed module conditions may be expressed in terms of action cores, and that when \SH\ does not hold, this approach involves an extra condition in terms of the ternary cosmash product.

\begin{defi}
\label{defi:internal crossed modules}
In a semi-abelian category $\A$ with \SH, an \emph{internal crossed module} is a pair $(X\xrightarrow{\partial}A,\xi)$ where $\partial\colon X\to A$ is a morphism in $\A$ and $\xi\colon A\flat X\to X$ is an internal action such that the diagram 
\[
\xymatrix{
X\flat X \ar[d]_-{\chi_X} \ar[r]^-{\partial\flat 1_X} \ar@{}[rd]|-{*_1} & A\flat X \ar@{}[rd]|-{*_2}\ar[d]^-{\xi} \ar[r]^-{1_A\flat\partial} & A\flat A \ar[d]^-{\chi_A} \\
X \ar@{=}[r] & X\ar[r]_-{\partial} & A
}
\]
commutes. $*_1$ is the \emph{Peiffer condition}, and $*_2$ the \emph{precrossed module condition}.
\end{defi}

\section{Compatible actions of groups}\label{Section groups}

\begin{defi}
\label{defi:coproduct action in grp}
Consider two groups $M$ and $N$ acting on each other via 
\begin{align*}
\xi^M_N\colon M\flat N\to N && \xi^N_M\colon N\flat M\to M
\end{align*}
and on themselves by conjugation. We are able to define induced actions $\xi^{M+N}_M$ and $\xi^{M+N}_N$ of the coproduct $M+N$ on $M$ and on $N$, in such a way that the following diagrams commute:
\begin{align}
\label{diag:coproduct action in grp is induced by the four actions}
\vcenter{
\xymatrix{
M\flat N \ar[r]^-{i_M\flat 1_N} \ar[rd]_-{\xi^M_N} & (M+N)\flat N \ar[d]^-{\xi^{M+N}_N}\\
& N}}
&&
\vcenter{
\xymatrix{
N\flat M \ar[r]^-{i_N\flat 1_M} \ar[rd]_-{\xi^N_M} & (M+N)\flat M \ar[d]^-{\xi^{M+N}_M}\\
& M}}
\\
\label{diag:coproduct action in grp is induced by the four actions'}
\vcenter{\xymatrix{
N\flat N \ar[r]^-{i_N\flat 1_N} \ar[rd]_-{\chi_N} & (M+N)\flat N \ar[d]^-{\xi^{M+N}_N}\\
& N
}}
&&
\vcenter{\xymatrix{
M\flat M \ar[r]^-{i_M\flat 1_M} \ar[rd]_-{\chi_M} & (M+N)\flat M \ar[d]^-{\xi^{M+N}_M}\\
& M
}}
\end{align}
This is done by defining the action $\xi^{M+N}_M\colon (M+N)\flat M\to M$ on the generators $s\overline{m}s^{-1}$ where $\overline{m}\in M$ and $s\in M+N$, inductively on the length of $s$:
\begin{equation}
\label{eq:defi of internal coproduct action in grp - CA}
\xi^{M+N}_M(s\overline{m}s^{-1})=
\begin{cases}
\overline{m} & \text{if $s$ is the empty word,}\\
\xi^{M+N}_M(s'\xi^N_M(n\overline{m}n^{-1})s'^{-1}) & \text{if $s=s'n$ with $n\in N$,}\\
\xi^{M+N}_M(s'\chi_M(m\overline{m}m^{-1})s'^{-1}) & \text{if $s=s'm$ with $m\in M$}\\
\end{cases}
\end{equation}
and similarly for $\xi^{M+N}_N$.
\end{defi}

\begin{rmk}
In particular we have that the equalities
\begin{align}
\label{eq:property of coproduct action in groups 1}
{}^{({}^{n}m)}m'&={({}^{n}m)}m'{({}^{n}m)}^{-1}={}^{n}(m{({}^{n^{-1}}m')}m^{-1})={}^{nmn^{-1}}m' \\
\label{eq:property of coproduct action in groups 2}
{}^{({}^{m}n)}n'&={({}^{m}n)}n'{({}^{m}n)}^{-1}={}^{m}(n{({}^{m^{-1}}n')}n^{-1})={}^{mnm^{-1}}n',
\end{align}
where the right-hand sides are given by the induced action of the coproduct, always hold. Diagrammatically this is expressed by the commutativity of
\begin{align}
\label{diag:property of coproduct action in grp}
\vcenter{
\xymatrixcolsep{3pc}
\xymatrix{
(N\flat M)\flat M \ar[r]^-{k_{N,M}\flat 1_M} \ar[d]_-{\xi^N_M\flat 1_M} & (M+N)\flat M \ar[d]^-{\xi^{M+N}_M}\\
M\flat M \ar[r]_-{\chi_M} & M,
}
}
&&
\vcenter{
\xymatrixcolsep{3pc}
\xymatrix{
(M\flat N)\flat N \ar[r]^-{k_{M,N}\flat 1_N} \ar[d]_-{\xi^M_N\flat 1_N} & (M+N)\flat N \ar[d]^-{\xi^{M+N}_N}\\
N\flat N \ar[r]_-{\chi_N} & N.
}
}
\end{align}
\end{rmk}

\begin{defi}
\label{defi:compatible actions in grp}
Two actions are said to be \emph{compatible} if also the equalities
\begin{equation}
\label{eq:equations for compatible actions in groups}
{}^{({}^{m}n)}m'={}^{mnm^{-1}}m' \qquad \text{and} \qquad {}^{({}^{n}m)}n'={}^{nmn^{-1}}n'	
\end{equation}
hold for each $m$, $m'\in M$ and $n$, $n'\in N$. If once again we examine these equalities from a diagrammatic point of view, then we see that they are equivalent to the commutativity of the diagrams
\begin{align}
\label{diag:axioms for compatible actions in grp}
\vcenter{
\xymatrixcolsep{3pc}
\xymatrix{
(M\flat N)\flat M \ar[r]^-{k_{M,N}\flat 1_M} \ar[d]_-{\xi^M_N\flat 1_M} & (M+N)\flat M \ar[d]^-{\xi^{M+N}_M}\\
N\flat M \ar[r]_-{\xi^N_M} & M,
}
}
&&
\vcenter{
\xymatrixcolsep{3pc}
\xymatrix{
(N\flat M)\flat N \ar[r]^-{k_{N,M}\flat 1_N} \ar[d]_-{\xi^N_M\flat 1_N} & (M+N)\flat N \ar[d]^-{\xi^{M+N}_N}\\
M\flat N \ar[r]_-{\xi^M_N} & N.
}
}
\end{align}
\end{defi}

A second look at these four equalities leads us to the following remark. 

\begin{rmk}
\label{rmk:from compatibility conditions in grp to Peiffer as a quotient}
The meaning of~\eqref{eq:property of coproduct action in groups 1} and~\eqref{eq:property of coproduct action in groups 2} is that for each $m\in M$ and $n\in N$
\begin{itemize}
 \item $({}^{n}m)nm^{-1}n^{-1}$ acts trivially on $M$,
 \item $({}^{m}n)mn^{-1}m^{-1}$ acts trivially on $N$;
\end{itemize}
whereas the meaning~\eqref{eq:equations for compatible actions in groups} is that for each $m\in M$ and $n\in N$
\begin{itemize}
 \item $({}^{n}m)nm^{-1}n^{-1}$ acts trivially on $N$;
 \item $({}^{m}n)mn^{-1}m^{-1}$ acts trivially on $M$.
\end{itemize}
If we define $K\leq M+N$ to be the normal closure of the subgroup generated by the elements of the form $({}^{n}m)nm^{-1}n^{-1}$ or $({}^{m}n)mn^{-1}m^{-1}$, we have that $K$ acts trivially on both $M$ and $N$ if and only if the two actions are compatible. 
\end{rmk}

The previous remark leads to the following definition given in~\cite{GH87}.

\begin{defi}
Given a pair of compatible actions as above, we define their \emph{Peiffer product} $M\bowtie N$ of $M$ and $N$ as the quotient
\[
\xymatrix{
K \ar@{{ |>}->}[r] & M+N \ar@{-{ >>}}[r]^-{q_K} & \frac{M+N}{K}\eqcolon M\bowtie N.
}
\]
\end{defi}

\begin{rmk}
\label{rmk:q_K=q}
Notice that the map $q_K$ and the Peiffer product $M\bowtie N$ can equivalently be defined as the coequaliser in the diagram
\begin{align}
\label{diag:defi of q in grp}
\xymatrixcolsep{4pc}
\xymatrix{
(N\flat M)+(M\flat N) \ar@<2pt>[r]^-{\binom{k_{N,M}}{k_{M,N}}} \ar@<-2pt>[r]_-{\xi^N_M+\xi^M_N} & M+N \ar@{-{ >>}}[r]^-{q} & M\bowtie N.
}
\end{align}
In order to explain why this definition is equivalent to the previous one, consider the map $q_K$ given by the first definition. It is easy to show that 
\[
\begin{cases}
q_K\circ i_M\circ \xi^N_M = q_K\circ k_{N,M}\\
q_K\circ i_N\circ \xi^M_N = q_K\circ k_{M,N}
\end{cases}
\]
since this is exactly what taking the quotient by $K$ means. But this is the same as saying
\[
\begin{cases}
q_K\circ (\xi^N_M+\xi^M_N) \circ i_{N\flat M} = q_K\circ k_{N,M}\\
q_K\circ (\xi^N_M+\xi^M_N) \circ i_{M\flat N} = q_K\circ k_{M,N}
\end{cases}
\]
which in turn is $q_K\circ (\xi^N_M+\xi^M_N)=q_K\circ \binom{k_{N,M}}{k_{M,N}}$. The universal property of the coequaliser is given by the universal property of the quotient by $K$ in a straightforward manner.
\end{rmk}

Since $K$ acts trivially on both $M$ and $N$ we can define induced actions $\xi^{M\bowtie N}_M$ and $\xi^{M\bowtie N}_N$ of $M \bowtie N$ on $M$ and $N$. They are such that the diagrams
\begin{align}
\label{diag:diagram involving action of coproduct and action of Peiffer product}
\vcenter{
\xymatrix{
(M+N)\flat M \ar[rd]_-{\xi^{M+N}_M} \ar[r]^-{q\flat 1_M} & (M\bowtie N)\flat M \ar[d]^-{\xi^{M\bowtie N}_M}\\
& M
}
}
&&
\vcenter{
\xymatrix{
(M+N)\flat N \ar[rd]_-{\xi^{M+N}_N} \ar[r]^-{q\flat 1_N} & (M\bowtie N)\flat N \ar[d]^-{\xi^{M\bowtie N}_N}\\
& M
}
}
\end{align}
commute. We can describe these actions of the Peiffer product through its universal property, but in order to do so, we need a preliminary remark.

\begin{rmk}
\label{rmk:the maps coequalised by the Peiffer product form a reflexive graph in grp}
The diagram
\[
\xymatrixcolsep{8pc}
\xymatrix{
(N\flat M)+(M\flat N) \ar@<1.5ex>[r]^-{\binom{k_{N,M}}{k_{M,N}}} \ar@<-1.5ex>[r]_-{\xi^N_M+\xi^M_N} & M+N \ar[l]|-{\eta^N_M+\eta^M_N}
}
\]
is a reflexive graph. Indeed, the composites $\binom{k_{N,M}}{k_{M,N}}\circ \eta^N_M+\eta^M_N$ and $\xi^N_M+\xi^M_N\circ \eta^N_M+\eta^M_N$ are equal to $1_{M+N}$: one is obvious and the other one is clear once we draw the diagram involved.
\end{rmk}

Lemma~\ref{lemma:-bX preserves coequalisers of reflexive graphs} implies that $q\flat 1_M$ is the coequaliser of $\binom{k_{N,M}}{k_{M,N}}\flat 1_M$ and $(\xi^N_M+\xi^M_N)\flat 1_M$ and that $q\flat 1_N$ is the coequaliser of $\binom{k_{N,M}}{k_{M,N}}\flat 1_N$ and $(\xi^N_M+\xi^M_N)\flat 1_N$. We want to use these universal properties to define induced actions $\xi^{M\bowtie N}_M$ and $\xi^{M\bowtie N}_N$ of $M \bowtie N$ on $M$ and $N$ as in Figure~\ref{Figure-Groups}. 
\begin{figure}
\begin{align*}
\xymatrixcolsep{3.5pc}
\xymatrix{
((N\flat M)+(M\flat N))\flat M \ar@<2pt>[r]^-{\binom{k_{N,M}}{k_{M,N}}\flat 1_M} \ar@<-2pt>[r]_-{\left(\xi^N_M+\xi^M_N\right)\flat 1_M} & (M+N)\flat M \ar[rd]_-{\xi^{M+N}_M} \ar[r]^-{q\flat 1_M} & (M\bowtie N)\flat M \ar@{.>}[d]^-{\xi^{M\bowtie N}_M}\\
& & M
}
\\
\xymatrixcolsep{3.5pc}
\xymatrix{
((N\flat M)+(M\flat N))\flat N \ar@<2pt>[r]^-{\binom{k_{N,M}}{k_{M,N}}\flat 1_N} \ar@<-2pt>[r]_-{\left(\xi^N_M+\xi^M_N\right)\flat 1_N} & (M+N)\flat N \ar[rd]_-{\xi^{M+N}_N} \ar[r]^-{q\flat 1_N} & (M\bowtie N)\flat N \ar@{.>}[d]^-{\xi^{M\bowtie N}_N}\\
& & N
}
\end{align*}	
\caption{Induced actions of the Peiffer product}\label{Figure-Groups}
\end{figure}
In order to do so, we need the next result.
\begin{prop}
The action $\xi^{M+N}_M$ coequalises $\binom{k_{N,M}}{k_{M,N}}\flat 1_M$ and $(\xi^N_M+\xi^M_N)\flat 1_M$. Similarly, the action $\xi^{M+N}_N$ coequalises $\binom{k_{N,M}}{k_{M,N}}\flat 1_N$ and $(\xi^N_M+\xi^M_N)\flat 1_N$. 
\end{prop}
\begin{proof}
Consider a generator $s\overline{m}s^{-1}$ of $((N\flat M)+(M\flat N))\flat M$ and write $s$ as juxtaposition of generators of $N\flat M$ and $M\flat N$, that is $s=s_1\cdots s_k$ with $s_j=n_jm_jn_j^{-1}\in N\flat M$ or $s_j=m_jn_jm_j^{-1}\in M\flat N$. We are going to prove the equality
\[
\xi^{M+N}_M\bigl(\bigl(\binom{k_{N,M}}{k_{M,N}}\flat 1_M\bigr)(s\overline{m}s^{-1})\bigr)=\xi^{M+N}_M\left(\left(\left(\xi^N_M+\xi^M_N\right)\flat 1_M\right)(s\overline{m}s^{-1})\right)
\]
by induction on $k$. First of all, notice that it is equivalent to the equality
\begin{equation}
\label{eq:the internal coproduct action coequalises things in grp}
\xi^{M+N}_M\left(s\overline{m}s^{-1}\right)=\xi^{M+N}_M\left(\epsilon(s)\overline{m}\epsilon(s)^{-1}\right)
\end{equation}
where $\epsilon(s)\coloneq (\xi^N_M+\xi^M_N)(s)\in M+N$.
In order to prove it when $s$ is the empty word, it suffices to notice that also $\epsilon(s)$ is the empty word. Now suppose we proved~\eqref{eq:the internal coproduct action coequalises things in grp} for each word whose decomposition involves at most $k-1$ generators of $N\flat M$ and $M\flat N$, consider $s=s_1\cdots s_k$ and denote $s'=s_1\cdots s_{k-1}$:
we have the chain of equalities
\begin{align*}
\xi^{M+N}_M\bigl(s\overline{m}s^{-1}\bigr)&=\xi^{M+N}_M\bigl(s's_k\overline{m}s_k^{-1}s'^{-1}\bigr)
=\xi^{M+N}_M\bigl(s'\bigl({}^{s_k}\overline{m}\bigr)s'^{-1}\bigr)\\
&=\xi^{M+N}_M\bigl(s'\bigl({}^{\epsilon(s_k)}\overline{m}\bigr)s'^{-1}\bigr)
=\xi^{M+N}_M\bigl(\epsilon(s')\bigl({}^{\epsilon(s_k)}\overline{m}\bigr)\epsilon(s')^{-1}\bigr)\\
&=\xi^{M+N}_M\bigl(\epsilon(s')\epsilon(s_k)\overline{m}\epsilon(s_k)^{-1}\epsilon(s')^{-1}\bigr)
=\xi^{M+N}_M\bigl(\epsilon(s)\overline{m}\epsilon(s)^{-1}\bigr)
\end{align*}
where
\[
\epsilon(s_k)=
\begin{cases}
{}^{n_k}m_k & \text{if $s_k=n_km_kn_k^{-1}\in N\flat M$,}\\
{}^{m_k}n_k & \text{if $s_k=m_kn_km_k^{-1}\in M\flat N$.}
\end{cases} 
\]
Finally we apply the same reasoning to $\xi^{M+N}_N$.
\end{proof}

\begin{prop}
\label{prop:Peiffer has two crossed module structures in grp}
We have two crossed module structures
\begin{align*}
(M\xrightarrow{l_M}M\bowtie N,\xi^{M\bowtie N}_M) && (N\xrightarrow{l_N}M\bowtie N,\xi^{M\bowtie N}_N)
\end{align*}
where the actions of the Peiffer product are induced as above and the maps $l_M$ and~$l_N$ are defined through
\begin{align}
\label{diag:defi of inclusions in Peiffer product}
\vcenter{
\xymatrix{
M \ar@{^{(}->}[rd]^-{i_M} \ar@/_/[rdd]_-{l_M} && N \ar@{_{(}->}[ld]_-{i_N} \ar@/^/[ldd]^-{l_N}\\
 & M+N \ar@{-{ >>}}[d]^-{q}\\
 & M\bowtie N
}
}
\end{align}
\end{prop}
\begin{proof}
We will show the claim only for $\xi^{M\bowtie N}_M$, since the proof in the other case uses the same strategy. We need to show the commutativity of the following squares
\[
\xymatrixcolsep{3pc}
\xymatrix{
M\flat M \ar[r]^-{\chi_M} \ar[d]_-{l_M\flat 1_M} & M \ar@{=}[d]\\
(M\bowtie N)\flat M \ar[r]^-{\xi^{M\bowtie N}_M} \ar[d]_-{1_{M\bowtie N}\flat l_M} & M \ar[d]^-{l_M}\\
(M\bowtie N)\flat (M\bowtie N) \ar[r]_-{\chi_{M\bowtie N}} & (M\bowtie N)
}
\]
For the commutativity of the upper square we have the chain of equalities
\begin{align*}
\xi^{M\bowtie N}_M \circ (l_M\flat 1_M)&=\xi^{M\bowtie N}_M \circ (q\flat 1_M)\circ (i_M\flat 1_M)
=\xi^{M+N}_M \circ (i_M\flat 1_M)
=\chi_M
\end{align*}
given by commutativity of diagrams~\eqref{diag:property of coproduct action in grp} and~\eqref{diag:diagram involving action of coproduct and action of Peiffer product}.

For what concerns the lower square, it can be shown to be commutative by direct calculations, using the explicit definition of the coproduct action given in~\eqref{eq:defi of internal coproduct action in grp - CA}. First of all we can precompose with the regular epimorphism $q\flat 1_M$: this entails that the required commutativity is equivalent to the equation 
\begin{equation}
\label{eq:equivalent condition for precrossed module condition for Peiffer in grp}
q\circ \chi_{M+N}\circ (1_{M+N}\flat i_M)=q\circ i_M\circ \xi^{M+N}_M. 
\end{equation}
Now we can take a word $s\in M+N$, an element $\overline{m}\in M$ and prove by induction on the length of $s$ that the general element $s\overline{m}s^{-1}\in (M+N)\flat M$ is sent by the two maps in~\eqref{eq:equivalent condition for precrossed module condition for Peiffer in grp} to
\begin{equation}
\label{eq:second equivalent condition for precrossed module condition for Peiffer in grp}
q\left({}^s\overline{m}\right)=q(s\overline{m}s^{-1}).
\end{equation}
Let us first show this equality for $s$ with length $0$, that is the empty word: we have that ${}^s\overline{m}=\overline{m}=s\overline{m}s^{-1}$ and hence~\eqref{eq:second equivalent condition for precrossed module condition for Peiffer in grp}. For the induction step we are going to use the equality
$q\left({}^n\overline{m}\right)=q(n\overline{m}n^{-1})$ coming from the definition of the Peiffer product. Suppose that~\eqref{eq:second equivalent condition for precrossed module condition for Peiffer in grp} holds for words $s$ with length $l(s)<k$. Given $s$ with length $k$ we can write it as $s=xs'$ with $x=m\in M$ or $x=n\in N$ and $l(s')=k-1$: now we have the chain of equalities
\begin{align*}
q\bigl({}^s\overline{m}\bigr)&=q\bigl({}^{xs'}\overline{m}\bigr)=q\bigl({}^{x}\bigl({}^{s'}\overline{m}\bigr)\bigr)
=q\bigl(x\bigl({}^{s'}\overline{m}\bigr)x^{-1}\bigr)=q(x)q\bigl({}^{s'}\overline{m}\bigr)q\bigl(x^{-1}\bigr)\\
&=q(x)q\bigl(s'\overline{m}s'^{-1}\bigr)q\bigl(x^{-1}\bigr)
=q\bigl(xs'\overline{m}s'^{-1}x^{-1}\bigr)
=q\bigl(s\overline{m}s^{-1}\bigr).
\end{align*}
We conclude that $(M\xrightarrow{l_M}M\bowtie N,\xi^{M\bowtie N}_M)$ and $(N\xrightarrow{l_N}M\bowtie N,\xi^{M\bowtie N}_N)$ are crossed modules.
\end{proof}

Furthermore we know that the actions $\xi^M_N$ and $\xi^N_M$ are in turn induced by $\xi^{M\bowtie N}_M$ and $\xi^{M\bowtie N}_N$ through the maps $l_M$ and $l_N$, that is
\begin{align*}
\xymatrix{
M\flat N \ar[r]^-{l_M\flat 1_N} \ar[rd]_-{\xi^M_N} & (M\bowtie N)\flat N \ar[d]^-{\xi^{M\bowtie N}_N}\\
& N
}
&&
\xymatrix{
N\flat M \ar[r]^-{l_N\flat 1_M} \ar[rd]_-{\xi^N_M} & (M\bowtie N)\flat M \ar[d]^-{\xi^{M\bowtie N}_M}\\
& M
}
\end{align*}
commute. This can be proved by using diagrams~\eqref{diag:coproduct action in grp is induced by the four actions}, \eqref{diag:coproduct action in grp is induced by the four actions'}, \eqref{diag:diagram involving action of coproduct and action of Peiffer product} and~\eqref{diag:defi of inclusions in Peiffer product}.

\begin{prop}[Remark~2.16 in~\cite{BL84}]
\label{prop:remark 2.16 in BL}
Two actions as above are compatible if and only if there exists a group $L$ and two crossed module structure $(M\xrightarrow{\mu}L,\xi^L_M)$ and $(N\xrightarrow{\nu}L,\xi^L_N)$ such that the actions of $M$ on $N$ and the action of $N$ on $M$ are induced from $L$ and its actions.
\end{prop}
\begin{proof}
$(\Leftarrow)$ We first show that the actions $\xi^M_N\coloneq \xi^L_N\circ(\mu\flat 1_N)$ and $\xi^N_M\coloneq \xi^L_M\circ(\nu\flat 1_M)$ are compatible. To see that they are actually actions it suffices to use Remark~\ref{rmk:xi(f flat 1) is an action}. In order to prove~\eqref{eq:equations for compatible actions in groups}---we will show only one of the two equalities, since the proof of the other follows the same steps---we are going to use the commutative diagrams induced from the crossed module structures involving $L$, that is
\begin{align*}
\xymatrix{
M\flat M \ar[r]^-{\chi_M} \ar[d]_-{\mu\flat 1_M} & M \ar@{=}[d]\\
L\flat M \ar[r]^-{\xi^L_M} \ar[d]_-{1_L\flat\mu} & M \ar[d]^-{\mu}\\
L\flat L \ar[r]_-{\chi_L} & L
}
&&
\xymatrix{
N\flat N \ar[r]^-{\chi_N} \ar[d]_-{\nu\flat 1_N} & N \ar@{=}[d]\\
L\flat N \ar[r]^-{\xi^L_N} \ar[d]_-{1_L\flat\nu} & N \ar[d]^-{\nu}\\
L\flat L \ar[r]_-{\chi_L} & L
}
\end{align*}
This gives us the chain of equalities
\begin{align*}
{}^{({}^{m}n)}m'&={}^{\nu({}^{\mu(m)}n)}m'={}^{\mu(m)\nu(n)\mu(m^{-1})}m'
={}^{\mu(m)}{}^{\nu(n)}\bigl({}^{\mu(m^{-1})}m'\bigr)\\
&={}^{\mu(m)}{}^{\nu(n)}\bigl({}^{m^{-1}}m'\bigr)
={}^{\mu(m)}\bigl({}^{\nu(n)}\bigl({}^{m^{-1}}m'\bigr)\bigr)
={}^{\mu(m)}\bigl({}^{n}\bigl({}^{m^{-1}}m'\bigr)\bigr)\\
&={}^{\mu(m)}\bigl({}^{nm^{-1}}m'\bigr)={}^{m}\bigl({}^{nm^{-1}}m'\bigr)
={}^{mnm^{-1}}m'.
\end{align*}

$(\Rightarrow)$ This implication is given by Proposition~\ref{prop:Peiffer has two crossed module structures in grp}.
\end{proof}

\section{Compatible actions in semi-abelian categories}\label{Section Semiabelian}
From now on we will consider $\A$ to be a semi-abelian category in which the condition \SH\ holds.

We are going to give a definition of compatible internal actions which is inspired by~\ref{defi:coproduct action in grp} and~\ref{defi:compatible actions in grp}, with some differences that we will explain here.

\begin{defi}
\label{defi:compatible internal actions}
Consider two objects $M$, $N\in\A$ which act on each other and on themselves by conjugation and denote the actions as
\begin{align*}
\chi_M&\colon M\flat M\to M & \chi_N&\colon N\flat N\to N\\
\xi^M_N&\colon M\flat N\to N & \xi^N_M&\colon N\flat M\to M.
\end{align*}
We say that the actions $\xi^M_N$ and $\xi^N_M$ are \emph{compatible} if there exist two actions 
\begin{align*}
\xi^{M+N}_N\colon (M+N)\flat N\to N && \xi^{M+N}_M\colon (M+N)\flat M\to M
\end{align*}
\lq\lq induced\rq\rq\ from $\xi^M_N$, $\xi^N_M$ and the conjugations, that is such that the diagrams~\eqref{diag:coproduct internal action is (uniquely) induced from the four actions} in Figure~\ref{Figure-CA0} as well as the diagrams~\eqref{diag:compatible internal actions axiom for M} and~\eqref{diag:compatible internal actions axiom for N} in Figure~\ref{Figure-CA.M-CA.N} commute.

\begin{figure}
\begin{align}
\label{diag:coproduct internal action is (uniquely) induced from the four actions}
&\vcenter{\xymatrix{
N\flat M \ar[r]^-{i_N\flat 1_M} \ar[rd]_-{\xi^N_M} & (M+N)\flat M \ar[d]^-{\xi^{M+N}_M}\\
& M
}}
&&
\vcenter{\xymatrix{
M\flat N \ar[r]^-{i_M\flat 1_N} \ar[rd]_-{\xi^M_N} & (M+N)\flat N \ar[d]^-{\xi^{M+N}_N}\\
& N
}}\nonumber
\\
&\vcenter{\xymatrix{
M\flat M \ar[r]^-{i_M\flat 1_M} \ar[rd]_-{\chi_M} & (M+N)\flat M \ar[d]^-{\xi^{M+N}_M}\\
& M
}}
&&
\vcenter{\xymatrix{
N\flat N \ar[r]^-{i_N\flat 1_N} \ar[rd]_-{\chi_N} & (M+N)\flat N \ar[d]^-{\xi^{M+N}_N}\\
& N
}}\tag{CA.0}\\
&\vcenter{\xymatrix{
M\diamond N\diamond M \ar[r]^-{j_{M,N,M}} \ar[d]_-{S^{N,M}_{1,2}} & (M+N)\flat M \ar[d]^-{\xi^{M+N}_M}\\
N\diamond M \ar[r]_-{{}^{\diamond}\xi^N_M} & M
}}
&&
\vcenter{\xymatrix{
M\diamond N\diamond N \ar[r]^-{j_{M,N,N}} \ar[d]_-{S^{M,N}_{1,2}} & (M+N)\flat N \ar[d]^-{\xi^{M+N}_N}\\
M\diamond N \ar[r]_-{{}^{\diamond}\xi^M_N} & N
}}\nonumber
\end{align}
\caption{The diagrams~\eqref{diag:coproduct internal action is (uniquely) induced from the four actions}}\label{Figure-CA0}
\end{figure}
\begin{figure}
\begin{align}\tag{CA.M}
\label{diag:compatible internal actions axiom for M}
\vcenter{
\xymatrixcolsep{6pc}
\xymatrix{
((N\flat M)+(M\flat N))\flat M \ar[r]^-{(\xi^N_M+\xi^M_N)\flat 1_M} \ar[d]_-{\binom{k_{N,M}}{k_{M,N}}\flat 1_M} & (M+N)\flat M \ar[d]^-{\xi^{M+N}_M}\\
(M+N)\flat M \ar[r]_-{\xi^{M+N}_M} & M
}
}\\
\tag{CA.N}
\label{diag:compatible internal actions axiom for N}
\vcenter{
\xymatrixcolsep{6pc}
\xymatrix{
((N\flat M)+(M\flat N))\flat N \ar[r]^-{(\xi^N_M+\xi^M_N)\flat 1_N} \ar[d]_-{\binom{k_{N,M}}{k_{M,N}}\flat 1_N} & (M+N)\flat N \ar[d]^-{\xi^{M+N}_N}\\
(M+N)\flat N \ar[r]_-{\xi^{M+N}_N} & N
}
}
\end{align}	
\caption{The diagrams~\eqref{diag:compatible internal actions axiom for M} and~\eqref{diag:compatible internal actions axiom for N}}\label{Figure-CA.M-CA.N}
\end{figure}
\end{defi}

This definition obviously implies the one given in the case of groups, but we will see later (Corollary~\ref{cor:in grp the internal defi of compatible action is the same as the particular one}) that in $\Grp$ the two definitions coincide.
The difference between these two definitions is twofold.
\begin{itemize}
 \item First of all, the commutativity of the two squares in~\eqref{diag:coproduct internal action is (uniquely) induced from the four actions} involving the ternary cosmash products is for free in $\Grp$ (and in $\LieAlg$ as one can see from~\cite{DiM19}). Right now it is not clear to us what are the conditions the category $\A$ must satisfy for the commutativity of these squares to be implied by the other four triangles in~\eqref{diag:coproduct internal action is (uniquely) induced from the four actions}. Note that this is quite similar to the ternary cosmash product conditions that appear in the description of internal crossed modules given in~\cite{HVdL11}.
 \item Likewise, note the difference between diagrams~\eqref{diag:compatible internal actions axiom for M} and~\eqref{diag:compatible internal actions axiom for N}, and their version for groups given by~\eqref{diag:axioms for compatible actions in grp}. The former two can be decomposed into the latter, together with~\eqref{diag:property of coproduct action in grp} and with additional conditions involving~$\flat$ and higher-order cosmash products. Also this aspect would benefit from further investigation.
\end{itemize}

\begin{rmk}
\label{rmk:uniqueness of the coproduct actions}
Notice that in the situation of the previous definition, the coproduct actions $\xi^{M+N}_M$ and $\xi^{M+N}_N$ are uniquely determined by the commutativities of~\eqref{diag:coproduct internal action is (uniquely) induced from the four actions} due to Lemma~\ref{lemma:covering the bemolle of a coproduct with some cosmash products}.
\end{rmk}

\begin{prop}
\label{prop:coterminal crossed module induce compatible actions}
Given a pair of coterminal crossed modules
\begin{align*}
(M\xrightarrow{\mu}L,\xi^L_M) && (N\xrightarrow{\nu}L,\xi^L_N)
\end{align*}
we can define actions $\xi^M_N$ and $\xi^N_M$ through the diagrams
\[
\xymatrix{
M\flat N \ar[rr]^-{\xi^M_N} \ar[rd]_-{\mu\flat 1_N} && N && N\flat M \ar[rr]^-{\xi^N_M} \ar[rd]_-{\nu\flat 1_M} && M\\
& L\flat N \ar[ru]_-{\xi^L_N} &&&& L\flat M \ar[ru]_-{\xi^L_M}
}
\]
These actions are then compatible in the sense of Definition~\ref{defi:compatible internal actions}.
\end{prop}
\begin{proof}
First of all, notice that $\xi^M_N$ and $\xi^N_M$ are actually actions due to Remark~\ref{rmk:xi(f flat 1) is an action}. Now, in order to show that they are compatible, we need to define the coproduct actions
\begin{align*}
\xi^{M+N}_N\colon (M+N)\flat N\to N && \xi^{M+N}_M\colon (M+N)\flat M\to M
\end{align*}
such that diagrams~\eqref{diag:coproduct internal action is (uniquely) induced from the four actions}, \eqref{diag:compatible internal actions axiom for M} and~\eqref{diag:compatible internal actions axiom for N} commute. These are defined as the compositions
\begin{align*}
\xymatrix{
(M+N)\flat M \ar[rr]^-{\xi^{M+N}_M} \ar[rd]_-{\binom{\mu}{\nu}\flat 1_M} && M\\
& L\flat M \ar[ru]_-{\xi^L_M}
}
&&
\xymatrix{
(M+N)\flat N \ar[rr]^-{\xi^{M+N}_N} \ar[rd]_-{\binom{\mu}{\nu}\flat 1_N} && N\\
& L\flat N \ar[ru]_-{\xi^L_N}
}
\end{align*}
Once again the fact that they are actions is given by Remark~\ref{rmk:xi(f flat 1) is an action}.
In order to show that the four triangles in~\eqref{diag:coproduct internal action is (uniquely) induced from the four actions} commute, we simply calculate
\begin{align*}
\xi^{M+N}_M\circ (i_M\flat 1_M)&=\xi^L_M\circ \left(\binom{\mu}{\nu}\flat 1_M\right)\circ (i_M\flat 1_M)=\xi^L_M\circ (\mu\flat 1_M)=\chi_M\\
\xi^{M+N}_M\circ (i_N\flat 1_M)&=\xi^L_M\circ \left(\binom{\mu}{\nu}\flat 1_M\right)\circ (i_N\flat 1_M)=\xi^L_M\circ (\nu\flat 1_M)=\xi^N_M\\
\xi^{M+N}_N\circ (i_M\flat 1_N)&=\xi^L_N\circ \left(\binom{\mu}{\nu}\flat 1_N\right)\circ (i_M\flat 1_N)=\xi^L_N\circ (\mu\flat 1_N)=\xi^M_N\\
\xi^{M+N}_N\circ (i_N\flat 1_N)&=\xi^L_N\circ \left(\binom{\mu}{\nu}\flat 1_N\right)\circ (i_N\flat 1_N)=\xi^L_N\circ (\nu\flat 1_N)=\chi_N
\end{align*}
using the crossed module conditions. For the first square in~\eqref{diag:coproduct internal action is (uniquely) induced from the four actions}, we use the diagrams
\begin{align*}
\xymatrix{
M\diamond N\diamond M \ar[d]_-{\mu\diamond\nu\diamond 1_M} \ar[r]^-{j_{M,N,M}} & (M+N)\flat M \ar[d]^-{(\mu+\nu)\flat 1_M}\\
L\diamond L\diamond M \ar[d]_-{S^{L,M}_{2,1}} \ar[r]^-{j_{L,L,M}} & (L+L)\flat M \ar[d]^-{\binom{1_L}{1_L}\flat 1_M}\\
L\diamond M \ar[r]_-{i_{L,M}} & L\flat M
}
&&
\xymatrix{
M\diamond N\diamond M \ar[d]_-{1_M \diamond\nu\diamond 1_M} \ar[r]^-{S^{N,M}_{1,2}} & N\diamond M \ar[d]_-{\nu\diamond 1_M} \ar[r]^-{{}^{\diamond}\xi^N_M} & M \ar@{=}[d]\\
M\diamond L\diamond M \ar[d]_-{\mu\diamond 1_L\diamond 1_M} \ar[r]_-{S^{L,M}_{1,2}} & L\diamond M \ar[r]_-{{}^{\diamond}\xi^L_M} & M \ar@{=}[d]\\
L\diamond L\diamond M \ar[r]_-{S^{L,M}_{2,1}} & L\diamond M \ar[r]_-{{}^{\diamond}\xi^L_M} & M
}
\end{align*}
induced by naturality and by the crossed module conditions (see Theorem~5.6 in~\cite{HVdL11}), in order to obtain the chain of equalities
\begin{align*}
\xi^{M+N}_M\circ j_{M,N,M} &=\xi^L_M\circ \left(\binom{\mu}{\nu}\flat 1_M\right)\circ j_{M,N,M}
=\xi^L_M\circ i_{L,M}\circ S^{L,M}_{2,1}\circ \mu\diamond\nu\diamond 1_M\\
&={}^{\diamond}\xi^N_M\circ S^{N,M}_{1,2}
\end{align*}
Through a similar reasoning we may prove the commutativity of the other square in~\eqref{diag:coproduct internal action is (uniquely) induced from the four actions}.
Finally, we need to show~\eqref{diag:compatible internal actions axiom for M}, that is the fact that $\xi^{M+N}_M$ coequalises the maps 
\[
\xymatrixcolsep{7pc}
\xymatrix{
((N\flat M)+(M\flat N))\flat M \ar@<2pt>[r]^-{\binom{k_{N,M}}{k_{M,N}}\flat 1_M} \ar@<-2pt>[r]_-{(\xi^N_M+\xi^M_N)\flat 1_M} & (M+N)\flat M.
}
\]
Here we have the chain of equalities
\begin{align*}
\xi^{M+N}_M\circ (\xi^N_M+\xi^M_N)\flat 1_M&= \xi^L_M\circ\binom{\mu}{\nu}\flat 1_M\circ\left(\left(\left(\xi^L_M\circ \nu\flat 1_M\right)+\left(\xi^L_N\circ \mu\flat 1_N\right)\right)\flat 1_M\right)\\
&=\xi^L_M\circ\left(\left(\binom{\mu}{\nu}\circ \left(\left(\xi^L_M\circ \nu\flat 1_M\right)+\left(\xi^L_N\circ \mu\flat 1_N\right)\right)\right)\flat 1_M\right)\\
&= \xi^L_M\circ\binom{\mu\circ\xi^L_M\circ \nu\flat 1_M}{\nu\circ\xi^L_N\circ \mu\flat 1_N}\flat 1_M\\
&= \xi^L_M\circ\binom{\chi_L\circ \nu\flat\mu}{\chi_L\circ \mu\flat\nu}\flat 1_M\\
&= \xi^L_M\circ\binom{\binom{\mu}{\nu}\circ k_{N,M}}{\binom{\mu}{\nu}\circ k_{M,N}}\flat 1_M\\
&= \xi^L_M\circ\left(\binom{\mu}{\nu}\flat 1_M\right)\circ \binom{k_{N,M}}{k_{M,N}}\flat 1_M\\
&= \xi^{M+N}_M\circ\bigl(\binom{k_{N,M}}{k_{M,N}}\flat 1_M\bigr).
\end{align*}
Through a similar reasoning we can show that~\eqref{diag:compatible internal actions axiom for N} commutes.
\end{proof}

We take the construction of the Peiffer product in~\eqref{diag:defi of q in grp} as a general definition.
\begin{defi}
Given two objects $M$ and $N$ acting on each other via $\xi^N_M$ and $\xi^M_N$, we define their \emph{Peiffer product} $M\bowtie N$ as the coequaliser
\begin{align}
\label{diag:defi of q}
\xymatrixcolsep{4pc}
\xymatrix{
(N\flat M)+(M\flat N) \ar@<2pt>[r]^-{\binom{k_{N,M}}{k_{M,N}}} \ar@<-2pt>[r]_-{\xi^N_M+\xi^M_N} & M+N \ar@{-{ >>}}[r]^-{q} & M\bowtie N.
}
\end{align}
\end{defi}

An equivalent definition of the Peiffer product of two actions can be given through the following proposition, which characterises it as the pushout of the two semi-direct products induced by the two actions.

\begin{prop}
\label{prop:Peiffer is a pushout of semidirect products}
Given a pair of actions $\xi^M_N\colon M\flat N\to N$ and $\xi^N_M\colon N\flat M\to M$ we can obtain the Peiffer product $M\bowtie N$ as the pushout
\begin{equation}
\label{diag:Peiffer as a pushout}
\vcenter{
\xymatrix{
M+N \ar[rd]|-{q} \ar@{-{ >>}}[r]^-{\sigma_{\xi^N_M}} \ar@{-{ >>}}[d]_-{\sigma_{\xi^M_N}} & M\rtimes N \ar@{-{ >>}}[d]^-{q_{M\rtimes N}}\\
N\rtimes M \ar@{-{ >>}}[r]_-{q_{N\rtimes M}} & M\bowtie N
}
}
\end{equation}
of the two semi-direct products.
\end{prop}
\begin{proof}
Recall that the semi-direct products are defined as the coequalisers
\begin{align*}
\xymatrixcolsep{4pc}
\xymatrixrowsep{3pc}
\xymatrix{
N\flat M \ar@<4pt>[r]^-{i_M\circ \xi^N_M} \ar@<-1pt>[r]_-{k_{N,M}} & M+N \ar@{-{ >>}}[r]^{\sigma_{\xi^N_M}} & M\rtimes N,
}
\\
\xymatrixcolsep{4pc}
\xymatrixrowsep{3pc}
\xymatrix{
M\flat N \ar@<4pt>[r]^-{ i_N\circ\xi^M_N} \ar@<-1pt>[r]_-{k_{M,N}} & M+N \ar@{-{ >>}}[r]^{\sigma_{\xi^M_N}} & N\rtimes M.
}
\end{align*}
By definition we know that $q$ coequalises each of these pairs of maps, and thus we obtain the unique regular epimorphisms $q_{N\rtimes M}$ and $q_{M\rtimes N}$ making the triangles
\begin{align*}
\xymatrix{
M+N \ar@{-{ >>}}[r]^-{\sigma_{\xi^N_M}} \ar@{-{ >>}}[rd]_-{q} & M\rtimes N \ar@{-{ >>}}[d]^-{q_{M\rtimes N}}\\
& M\bowtie N
}
&&
\xymatrix{
M+N \ar@{-{ >>}}[r]^-{\sigma_{\xi^M_N}} \ar@{-{ >>}}[rd]_-{q} & N\rtimes M \ar@{-{ >>}}[d]^-{q_{N\rtimes M}}\\
& M\bowtie N
}
\end{align*}
commute.
Now in order to prove that~\eqref{diag:Peiffer as a pushout} is a pushout, suppose there exist $f\colon {M\rtimes N\to Z}$ and $g\colon {N\rtimes M\to Z}$ such that $\gamma\coloneq f\circ\sigma_{\xi^N_M}=g\circ\sigma_{\xi^M_N}$. It suffices to prove that $\gamma$ coequalises the maps defining $q$: 
\begin{align*}
\gamma\circ\binom{k_{N,M}}{k_{M,N}}&=\binom{\gamma\circ k_{N,M}}{\gamma\circ k_{M,N}}=\binom{f\circ\sigma_{\xi^N_M}\circ k_{N,M}}{g\circ\sigma_{\xi^M_N}\circ k_{M,N}}
=\binom{f\circ\sigma_{\xi^N_M}\circ i_M\circ\xi^N_M}{g\circ\sigma_{\xi^M_N}\circ i_N\circ\xi^M_N}=\binom{\gamma\circ i_M\circ\xi^N_M}{\gamma\circ i_N\circ\xi^M_N}\\
&=\gamma\circ(\xi^N_M+\xi^M_N).
\end{align*}
This gives us a unique map $\gamma'\colon M\bowtie N\to Z$ such that $\gamma'\circ q_{M\rtimes N}=f$ and $\gamma'\circ q_{N\rtimes M}=g$ because $\sigma_{\xi^N_M}$ and $\sigma_{\xi^M_N}$ are epimorphisms.
\end{proof}

The idea behind the Peiffer product $M\bowtie N$ is that it should be the universal object acting on $M$ and $N$ with two crossed modules structures, as soon as these two objects act on each other compatibly. This is meant to solve the following problem. If we are in the situation of two compatible actions, we have induced coproduct actions whose precrossed module conditions
\begin{align}
\label{diag:precrossed module conditions for coproduct actions}
\vcenter{
\xymatrix{
(M+N)\flat M \ar[d]_-{1_{M+N}\flat i_M} \ar[r]^-{\xi^{M+N}_M} & M \ar[d]^-{i_M}\\
(M+N)\flat (M+N) \ar[r]_-{\chi_{M+N}} & M+N
}
}
&&
\vcenter{
\xymatrix{
(M+N)\flat N \ar[d]_-{1_{M+N}\flat i_N} \ar[r]^-{\xi^{M+N}_N} & N \ar[d]^-{i_N}\\
(M+N)\flat (M+N) \ar[r]_-{\chi_{M+N}} & M+N
}
}
\end{align}
are generally not satisfied. (However, the Peiffer conditions always hold.)

Hence we want to do two things: we want to define actions of the Peiffer product on $M$ and $N$ induced from the coproduct actions, and then we want to show that the postcomposition with the quotient defining the Peiffer product makes the previous squares commute, so that we obtain two crossed module structures.

Again by using Lemma~\ref{lemma:-bX preserves coequalisers of reflexive graphs} and Remark~\ref{rmk:the maps coequalised by the Peiffer product form a reflexive graph in grp} we deduce that in order to define the actions $\xi^{M\bowtie N}_M$ and $\xi^{M\bowtie N}_N$ of the Peiffer product as in Figure~\ref{Figure:Peiffer} (compare with the group case, Figure~\ref{Figure-Groups}) 
\begin{figure}
\begin{align}
\label{diag:defi of Peiffer action on M}
\vcenter{
\xymatrixcolsep{3.5pc}
\xymatrix{
((N\flat M)+(M\flat N))\flat M \ar@<2pt>[r]^-{\binom{k_{N,M}}{k_{M,N}}\flat 1_M} \ar@<-2pt>[r]_-{\left(\xi^N_M+\xi^M_N\right)\flat 1_M} & (M+N)\flat M \ar[rd]_-{\xi^{M+N}_M} \ar[r]^-{q\flat 1_M} & (M\bowtie N)\flat M \ar@{.>}[d]^-{\xi^{M\bowtie N}_M}\\
& & M
}
}
\\
\label{diag:defi of Peiffer action on N}
\vcenter{
\xymatrixcolsep{3.5pc}
\xymatrix{
((N\flat M)+(M\flat N))\flat N \ar@<2pt>[r]^-{\binom{k_{N,M}}{k_{M,N}}\flat 1_N} \ar@<-2pt>[r]_-{\left(\xi^N_M+\xi^M_N\right)\flat 1_N} & (M+N)\flat N \ar[rd]_-{\xi^{M+N}_N} \ar[r]^-{q\flat 1_N} & (M\bowtie N)\flat N \ar@{.>}[d]^-{\xi^{M\bowtie N}_N}\\
& & N
}
}
\end{align}	
\caption{Actions $\xi^{M\bowtie N}_M$ and $\xi^{M\bowtie N}_N$ of the Peiffer product}\label{Figure:Peiffer}
\end{figure}
it suffices to show that $\xi^{M+N}_M$ coequalises the parallel maps in~\eqref{diag:defi of Peiffer action on M} and that $\xi^{M+N}_N$ coequalises the parallel maps in~\eqref{diag:defi of Peiffer action on N}. These conditions are equivalent to the commutativity of~\eqref{diag:compatible internal actions axiom for M} and~\eqref{diag:compatible internal actions axiom for N}.

Now we have the desired actions of the Peiffer product, but in order to obtain the crossed module structures we need to show that postcomposing with the quotient~$q$ makes the diagrams~\eqref{diag:precrossed module conditions for coproduct actions} commute. In fact we are going to prove more than this: the Peiffer product is the coequaliser of those maps!

\begin{defi}
Given a pair of compatible actions $\xi^N_M$ and $\xi^M_N$, we define the \emph{strong Peiffer
 product} $M\bowtie_S N$ as the coequaliser in the diagram
\begin{align}
\label{diag:defi of q_S}
{
\xymatrixcolsep{2.8pc}
\xymatrix{
((M+N)\flat M)+((M+N)\flat N) \ar@<2pt>[rr]^-{\chi_{M+N}\circ \binom{1_{M+N}\flat i_M}{1_{M+N}\flat i_N}} \ar@<-2pt>[rr]_-{\xi^{M+N}_M+\xi^{M+N}_N} && M+N \ar@{-{ >>}}[r]^-{q_S} & M\bowtie_S N.
}
}
\end{align}
\end{defi}

\begin{rmk}
\label{rmk:why q_S and q}
It is important to notice that in principle there is a huge difference between the coequaliser in~\eqref{diag:defi of q} and the one in~\eqref{diag:defi of q_S}:
\begin{itemize}
 \item the latter makes sense only if the two actions are already compatible---otherwise the existence of the coproduct actions is not guaranteed; by definition, the strong Peiffer product coequalises the compositions in~\eqref{diag:precrossed module conditions for coproduct actions};
 \item the former makes sense even when the two actions are not compatible; it is obtained following the ideas from the particular compatibility conditions in the case of $\Grp$ through Remark~\ref{rmk:from compatibility conditions in grp to Peiffer as a quotient} and Remark~\ref{rmk:q_K=q}.
\end{itemize}
This means that by taking~\eqref{diag:defi of q} as a definition of $M\bowtie N$, we would not immediately have that the Peiffer product is the universal way to coequalise the compositions in~\eqref{diag:precrossed module conditions for coproduct actions}. Obviously if we precompose the maps in~\eqref{diag:defi of q_S} with $(i_N\flat 1_M)+(i_M\flat 1_N)$, we see that $q_S$ coequalises also the maps defining $q$
\[
q_S\circ (\xi^N_M+\xi^M_N) = q_S\circ\binom{k_{N,M}}{k_{M,N}} 
\]
but for the converse we need the following proposition.
\end{rmk}

\begin{prop}
\label{prop:q=q_S}
Consider two actions $\xi^M_N$ and $\xi^N_M$ which are compatible in the sense of Definition~\ref{defi:compatible internal actions}. Then the Peiffer product $M\bowtie N$ as in~\eqref{diag:defi of q} and the strong Peiffer product $M\bowtie_S N$ as in~\eqref{diag:defi of q_S} are isomorphic. 
\end{prop}
\begin{proof}
In order to obtain the needed isomorphism it suffices to show that $q$ coequalises the maps defining $q_S$: since the converse already holds due to Remark~\ref{rmk:why q_S and q}, we obtain the claim by the universal properties of the coequalisers. Recalling Lemma~\ref{lemma:covering the bemolle of a coproduct with some cosmash products} we just need to show that $q$ coequalises the two compositions in
\[
\xymatrixrowsep{3pc}
\xymatrixcolsep{0pc}
\xymatrix{
(M\flat M)+(N\flat M)+(M\diamond N\diamond M)+(M\flat N)+(N\flat N)+(M\diamond N\diamond N) \ar@{-{ >>}}[d]_-{
\scalebox{0.7}{$
\left\lgroup
\begin{array}{c}
i_M\flat 1_M\\
i_N\flat 1_M\\
j_{M,N,M}
\end{array}
\right\rgroup+\left\lgroup
\begin{array}{c}
i_M\flat 1_N\\
i_N\flat 1_N\\
j_{M,N,N}
\end{array}
\right\rgroup
$}
}\\
((M+N)\flat M)+((M+N)\flat N) \ar@<2pt>[rr]^-{\chi_{M+N}\circ\binom{1_{M+N}\flat i_M}{1_{M+N}\flat i_N}} \ar@<-2pt>[rr]_-{\xi^{M+N}_M+\xi^{M+N}_N} && M+N
}
\]
By the universal property of the coproduct we can consider each component separately and since the last three are similar to the first three---it suffices to exchange $M$ and $N$---we are going to examine only the first three.
\begin{itemize}
\item Precomposing with the inclusion of $M\flat M$, we obtain
\begin{align*}
q\circ\chi_{M+N}\binom{1_{M+N}\flat i_M}{1_{M+N}\flat i_N}\circ i_1 \circ (i_M\flat 1_M)&=q\circ \chi_{M+N}\circ (i_M\flat i_M)\\
&=q\circ i_M\circ \chi_M\\
&=q\circ i_M\circ\xi^{M+N}_M\circ (i_M\flat 1_M)\\
&=q\circ \left(\xi^{M+N}_M+\xi^{M+N}_N\right)\circ i_1\circ (i_M\flat 1_M).
\end{align*}
\item Precomposing with the inclusion of $N\flat M$, and using the definition of $q$ we obtain
\begin{align*}
q\circ\chi_{M+N}\binom{1_{M+N}\flat i_M}{1_{M+N}\flat i_N}\circ i_1 \circ (i_N\flat 1_M)&=q\circ \chi_{M+N}\circ (i_N\flat i_M)\\
&=q\circ k_{N,M}\\
&=q\circ i_M\circ\xi^N_M\\
&=q\circ i_M\circ\xi^{M+N}_M\circ (i_N\flat 1_M)\\
&=q\circ \left(\xi^{M+N}_M+\xi^{M+N}_N\right)\circ i_1\circ (i_N\flat 1_M).
\end{align*}
\item Precomposing with the inclusion of $M\diamond N\diamond M$, we obtain
\begin{align*}
q\circ\chi_{M+N}\binom{1_{M+N}\flat i_M}{1_{M+N}\flat i_N}\circ i_1 \circ j_{M,N,M}&=q\circ \chi_{M+N}\circ (1_{M+N}\flat i_M)\circ j_{M,N,M}\\
&=q\circ h_{N,M}\circ S^{N,M}_{1,2}\\
&=q\circ k_{N,M}\circ i_{N,M}\circ S^{N,M}_{1,2}\\
&=q\circ k_{N,M}\circ i_{N,M}\circ S^{N,M}_{1,2}\\
&=q\circ i_M\circ\xi^N_M\circ i_{N,M}\circ S^{N,M}_{1,2}\\
&=q\circ i_M\circ\xi^{M+N}_M\circ j_{M,N,M}\\
&=q\circ \left(\xi^{M+N}_M+\xi^{M+N}_N\right)\circ i_1\circ j_{M,N,M}.\qedhere
\end{align*}
\end{itemize}
\end{proof}

This means that $M\bowtie_S N\cong M\bowtie N$ and that $q$ is the universal map making~\eqref{diag:precrossed module conditions for coproduct actions} commute through postcomposition.

Our aim now is to show that $\xi^{M\bowtie N}_M$ and $\xi^{M\bowtie N}_N$ are indeed actions, which moreover induce two crossed module structures.

\begin{prop}
\label{prop:crossed module structures of the Peiffer product}
The maps $\xi^{M\bowtie N}_M$ and $\xi^{M\bowtie N}_N$ are internal actions. We have crossed module structures
\begin{align*}
(M\xrightarrow{l_M}M\bowtie N,\xi^{M\bowtie N}_M) && (N\xrightarrow{l_N}M\bowtie N,\xi^{M\bowtie N}_N)
\end{align*}
where the maps $l_M$ and $l_N$ are defined as in~\eqref{diag:defi of inclusions in Peiffer product}. Further, as in Proposition~\ref{prop:coterminal crossed module induce compatible actions}, the compatible actions induced by these crossed module structures coincide with the actions $\xi^N_M$ and $\xi^M_N$.
\end{prop}
\begin{proof}
We are going to prove the claim only for $\xi^{M\bowtie N}_M$ and $l_M$, since the reasoning can be repeated for $\xi^{M\bowtie N}_N$ and $l_N$.

In order to see that $\xi^{M\bowtie N}_M$ is automatically an action it suffices to follow these steps:
\begin{itemize}
\item show that $\eta^{M\bowtie N}_M=(q\flat 1_M)\circ \eta^{M+N}_M$ using the diagram
\[
\xymatrix{
M \ar@{^{(}->}[rd]^-{i_M} \ar[d]_-{\eta^{M+N}_M}\\
(M+N)\flat M \ar@{{ |>}->}[r]^-{k_{M+N,M}} \ar[d]_-{q\flat 1_M} & (M+N)+M \ar[d]^-{q+1_M}\\
(M\bowtie N)\flat M \ar@{{ |>}->}[r]^-{k_{M\bowtie N,M}} & (M\bowtie N)+M;
}
\]
\item show the first axiom:
\[
\xi^{M\bowtie N}_M\circ\eta^{M\bowtie N}_M=\xi^{M\bowtie N}_M\circ(q\flat 1_M)\circ \eta^{M+N}_M=\xi^{M+N}_M\circ \eta^{M+N}_M=1_M;
\]
\item $q\flat(q\flat 1_M)$ is a regular epimorphism due to Lemma~\ref{lemma:regular epis are bemolle-stable};
\item show the second axiom 
\[
\xi^{M\bowtie N}_M\circ\mu^{M\bowtie N}_M=\xi^{M\bowtie N}_M\circ(1_{M\bowtie N}\flat \xi^{M\bowtie N}_M)
\]
using the commutativity of the outer rectangle
\[
\xymatrix{
(M+N)\flat((M+N)\flat M) \ar[r]^-{\mu^{M+N}_M} \ar@{-{ >>}}[d]_-{q\flat(q\flat 1_M)} & (M+N)\flat M \ar@{-{ >>}}[d]^-{q\flat 1_M}\\
(M\bowtie N)\flat((M\bowtie N)\flat M) \ar[r]^-{\xi^{M\bowtie N}_M} \ar[d]_-{1_{M\bowtie N}\flat \xi^{M\bowtie N}_M} & (M\bowtie N)\flat M \ar[d]^-{\xi^{M\bowtie N}_M}\\
(M\bowtie N)\flat M \ar[r]_-{\xi^{M\bowtie N}_M} & M
}
\]
given by the second axiom for the action $\xi^{M+N}_M$.
\end{itemize}
It remains to prove that $(M\xrightarrow{l_M}M\bowtie N,\xi^{M\bowtie N}_M)$ is indeed a crossed module. This amounts to the commutativity of the squares
\[
\xymatrix{
M\flat M \ar[r]^-{\chi_M} \ar[d]_-{l_M\flat 1_M} & M \ar@{=}[d]\\
(M\bowtie N)\flat M \ar[r]^-{\xi^{M\bowtie N}_M} \ar[d]_-{1_{M\bowtie N}\flat l_M} & M \ar[d]^-{l_M}\\
(M\bowtie N)\flat (M\bowtie N) \ar[r]_-{\chi_{M\bowtie N}} & (M\bowtie N).
}
\]
For the upper square we have the chain of equalities
\begin{align*}
\xi^{M\bowtie N}_M \circ (l_M\flat 1_M)&=\xi^{M\bowtie N}_M \circ (q\flat 1_M)\circ (i_M\flat 1_M)
=\xi^{M+N}_M \circ (i_M\flat 1_M)
=\chi_M.
\end{align*}
In order to show that the lower square commutes, consider the diagram
\[
\xymatrix{
(M+N)\flat M \ar@{-{ >>}}[d]_-{q\flat 1_M} \ar@/^/[rd]^-{\xi^{M+N}_M}\\
(M\bowtie N)\flat M \ar[r]^-{\xi^{M\bowtie N}_M} \ar[d]_-{1_{M\bowtie N}\flat l_M} & M \ar[d]^-{l_M}\\
(M\bowtie N)\flat (M\bowtie N) \ar[r]_-{\chi_{M\bowtie N}} & (M\bowtie N).
}
\]
Since $q\flat 1_M$ is a regular epimorphism, it suffices to prove that the outer diagram is commutative. We decompose it as
\[
\xymatrix{
(M+N)\flat M \ar[r]^-{\xi^{M+N}_M} \ar[d]_-{1_{M+N}\flat i_M} & M \ar[d]^-{i_M}\\
(M+N)\flat (M+N) \ar[d]_-{q\flat q} \ar[r]_-{\chi_{M+N}} & M+N \ar[d]^-{q}\\
(M\bowtie N)\flat (M\bowtie N) \ar[r]_-{\chi_{M\bowtie N}} & M\bowtie N\\
}
\]
It is easy to check that the lower square commutes and thanks to this, by using Proposition~\ref{prop:q=q_S}, we find that the whole rectangle commutes.

Finally we know that the actions $\xi^M_N$ and $\xi^N_M$ are in turn induced by $\xi^{M\bowtie N}_M$ and $\xi^{M\bowtie N}_N$ through the maps $l_M$ and $l_N$, that is
\begin{align*}
\vcenter{\xymatrix{
M\flat N \ar[r]^-{l_M\flat 1_N} \ar[rd]_-{\xi^M_N} & (M\bowtie N)\flat N \ar[d]^-{\xi^{M\bowtie N}_N}\\
& N
}}
\qquad\text{and}\qquad
\vcenter{\xymatrix{
N\flat M \ar[r]^-{l_N\flat 1_M} \ar[rd]_-{\xi^N_M} & (M\bowtie N)\flat M \ar[d]^-{\xi^{M\bowtie N}_M}\\
& M
}}
\end{align*}
commute. This can be proved by using the definition of $l_M$ and $l_N$ and the commutativity of diagrams~\eqref{diag:coproduct internal action is (uniquely) induced from the four actions}, \eqref{diag:defi of Peiffer action on M} and~\eqref{diag:defi of Peiffer action on N}.
\end{proof}

\begin{rmk}
Notice that in the previous proposition we are implicitly using the \SH\ condition: indeed we are using Definition~\ref{defi:internal crossed modules}, which requires \SH, as a definition for internal crossed modules.
\end{rmk}

Combining Proposition~\ref{prop:coterminal crossed module induce compatible actions} and Proposition~\ref{prop:crossed module structures of the Peiffer product} we obtain the following characterisation of compatible actions, the main result of this article:

\begin{thm}
\label{thm:compatible actions iff pair of coterminal crossed module}
In a semi-abelian category that satisfies \SH, two actions $\xi^M_N$ and $\xi^N_M$ are compatible if and only if there exists an object $L$ endowed with crossed module structures
\begin{align*}
(M\xrightarrow{\mu}L,\xi^L_M) && (N\xrightarrow{\nu}L,\xi^L_N)
\end{align*}
which via the commutative triangles
\[
\xymatrix{
M\flat N \ar[rr]^-{\xi^M_N} \ar[rd]_-{\mu\flat 1_N} && N && N\flat M \ar[rr]^-{\xi^N_M} \ar[rd]_-{\nu\flat 1_M} && M\\
& L\flat N \ar[ru]_-{\xi^L_N} &&&& L\flat M \ar[ru]_-{\xi^L_M}
}
\]
induce the given actions.\noproof
\end{thm}

As a consequence, our definition of compatible internal actions is indeed an extension of the particular definitions for groups and Lie algebras.

\begin{cor}
\label{cor:in grp the internal defi of compatible action is the same as the particular one}
In $\Grp$ Definition~\ref{defi:compatible internal actions} coincides with Definition~\ref{defi:compatible actions in grp}.
\end{cor}
\begin{proof}
This is a combination of Theorem~\ref{thm:compatible actions iff pair of coterminal crossed module} with Proposition~\ref{prop:remark 2.16 in BL}.
\end{proof}

\begin{cor}
\label{cor:in LieAlg the internal defi of compatible action is the same as the particular one}
The definition of compatible actions of Lie algebras given in~\cite{Ell91} coincides with Definition~\ref{defi:compatible internal actions} restricted to the category $\LieAlg$.
\end{cor}
\begin{proof}
This is obtained through Theorem~\ref{thm:compatible actions iff pair of coterminal crossed module} by using Theorem~2.17 from~\cite{DiM19}.
\end{proof}

\section{Universal properties of the Peiffer product}\label{Section Peiffer}
The Peiffer product $M\bowtie N$ is the universal way to associate a coterminal pair of crossed modules to a pair of compatible actions.

\begin{prop}\label{Universal property}
Consider a pair of compatible actions $\xi^M_N$ and $\xi^N_M$ and the pairs of coterminal crossed modules inducing them. The pair given by the Peiffer product is the universal one, in the sense that it is initial: for any pair of crossed modules 
\begin{align*}
(M\xrightarrow{\mu}L,\xi^L_M) && (N\xrightarrow{\nu}L,\xi^L_N)
\end{align*}
inducing $\xi^M_N$ and $\xi^N_M$ there exists a unique morphism $\verticalbinomial{\mu}{\nu}\colon M\bowtie N\to L$ making the diagram
\[
\xymatrix{
& M \ar[d]_-{l_M} \ar@/^/[ddr]^-{\mu}\\
N \ar[r]^-{l_N} \ar@/_/[drr]_-{\nu} & M\bowtie N \ar@{.>}[rd]|-{\verticalbinomial{\mu}{\nu}}\\
&& L
}
\]
commute.
\end{prop}
\begin{proof}
It suffices to show that $\binom{\mu}{\nu}\colon M+N\to L$ coequalises the two maps defining $M\bowtie N$. Indeed that would give us a unique map $\verticalbinomial{\mu}{\nu}$ such that
\[
\xymatrix{
M+N \ar[r]^-{q} \ar[rd]_-{\binom{\mu}{\nu}} & M\bowtie N \ar@{.>}[d]^-{\verticalbinomial{\mu}{\nu}}\\
& L
}
\]
and then by precomposing with the inclusion we would get
\begin{align*}
\mu=\binom{\mu}{\nu}\circ i_M=\verticalbinomial{\mu}{\nu}\circ q\circ i_M=\verticalbinomial{\mu}{\nu}\circ l_M,\quad
\nu=\binom{\mu}{\nu}\circ i_N=\verticalbinomial{\mu}{\nu} \circ q\circ i_N=\verticalbinomial{\mu}{\nu}\circ l_N.
\end{align*}
Therefore we have to show that the two compositions
\[
\xymatrixcolsep{3pc}
\xymatrix{
(N\flat M)+(M\flat N) \ar@<2pt>[rr]^-{\binom{k_{N,M}}{k_{M,N}}} \ar@<-2pt>[rr]_-{\xi^N_M+\xi^M_N} && M+N \ar[r]^-{\binom{\mu}{\nu}} & L
}
\]
are equal. This is done via the chain of equalities
\[
\binom{\mu}{\nu}\circ (\xi^N_M+\xi^M_N)= \binom{\mu\circ\xi^L_M\circ \nu\flat 1_M}{\nu\circ\xi^L_N\circ \mu\flat 1_N}=\binom{\binom{\mu}{\nu}\circ k_{N,M}}{\binom{\mu}{\nu}\circ k_{M,N}}=\binom{\mu}{\nu}\circ \binom{k_{N,M}}{k_{M,N}}.\qedhere
\]
\end{proof}

\begin{lemma}
Consider two pairs of coterminal crossed modules
\begin{align*}
\xymatrix{
& M\ar[d]^-{\mu}\\
N \ar[r]_-{\nu}& L
}
&& 
\xymatrix{
& M\ar[d]^-{\mu'}\\
N \ar[r]_-{\nu'}& L'
}
\end{align*}
such that they induce the same actions between $M$ and $N$, that is such that the diagrams 
\begin{align*}
\xymatrix{
N\flat M \ar[rd]|-{\xi^N_M} \ar[r]^-{\nu\flat 1_M} \ar[d]_-{\nu'\flat 1_M} & L\flat M \ar[d]^-{\xi^L_M}\\
L'\flat M \ar[r]_-{\psi'_M} & M
}
&&
\xymatrix{
M\flat N \ar[rd]|-{\xi^M_N} \ar[r]^-{\mu\flat 1_N} \ar[d]_-{\mu'\flat 1_N} & L\flat N \ar[d]^-{\xi^L_N}\\
L'\flat N \ar[r]_-{\psi'_N} & N
}
\end{align*}
commute. Up to isomorphism, they induce the same Peiffer product $M\bowtie N$.
\end{lemma}
\begin{proof}
The induced actions $\xi^{M+N}_M$ and $\xi'^{M+N}_M$ (resp.\ $\xi^{M+N}_N$ and $\xi'^{M+N}_N$) coincide when restricted to $M\flat M$, $N\flat M$ and $M\diamond N\diamond M$ (resp.\ $M\flat N$, $N\flat N$ and $M\diamond N\diamond N$), therefore it suffices to use Remark~\ref{rmk:uniqueness of the coproduct actions} to obtain that $\xi^{M+N}_M=\xi'^{M+N}_M$ (resp.\ $\xi^{M+N}_N=\xi'^{M+N}_N$). As a consequence they induce isomorphic Peiffer products and isomorphic crossed module structures.
\end{proof}

Finally we use Proposition~\ref{prop:Peiffer is a pushout of semidirect products} to show the link between our definition of Peiffer product and the one given in~\cite{CMM17}.

\begin{rmk}
We know from Proposition~3.2 in~\cite{CMM17} that, as soon as $(M\xrightarrow{\mu}L,\xi^L_M)$ and $(N\xrightarrow{\nu}L,\xi^L_N)$ are (pre)crossed modules, we have induced actions of $L$ on~$M\rtimes N$ and~$N\rtimes M$ with corresponding (pre)crossed module structures. In general this is not true for $M\bowtie N$, but if $\A$ is algebraically coherent, by Proposition~4.1 and Proposition~4.3 in~\cite{CMM17}, and by Proposition~\ref{prop:Peiffer is a pushout of semidirect products} we obtain that our definition of Peiffer product coincides with the one given by Cigoli, Mantovani and Metere: consequently $M\bowtie N$ is endowed with a precrossed module structure ${(\verticalbinomial{\mu}{\nu}\colon M\bowtie N\to L,\psi_{M\bowtie N})}$ as soon as $M$ and $N$ are so. Finally, when $\A$ satisfies the condition (UA) as well (see~\cite{CMM17} for more on this condition), Theorem~5.2 in~\cite{CMM17} tells us that the Peiffer product precrossed module turns out to be a crossed module as soon as $M$ and $N$ are so. Actually then it is the coproduct of $(M\xrightarrow{\mu}L,\xi^L_M)$ and $(N\xrightarrow{\nu}L,\xi^L_N)$ in the category $\XMod_L(\A)$ of $L$-crossed modules in $\A$.	
\end{rmk}

\begin{rmk}
We do not know whether $L$ acts on $M\bowtie N$ when $\A$ is not algebraically coherent. And even if so, it is not clear to us whether this action defines a precrossed module structure.
\end{rmk}


\end{document}